\documentclass[12pt]{amsart}
\usepackage{txfonts}
\usepackage{bbm}
\usepackage{mathrsfs}
\usepackage{amssymb}

\usepackage{amssymb,amsmath,amsthm,amscd}

\usepackage[all]{xy}

\numberwithin{equation}{section}

\newtheoremstyle{fancy1}{10pt}{10pt}{\itshape}{12pt}{\textsc\bgroup}{.\egroup}{8pt}{}
\newtheoremstyle{fancy2}{10pt}{10pt}{}{12pt}{\itshape}{.}{8pt}{ }

\theoremstyle{fancy1}

\newtheorem{cor}[equation]{Corollary}
\newtheorem{lem}[equation]{Lemma}

\newtheorem{prop}[equation]{Proposition}
\newtheorem{thm}[equation]{Theorem}
\newtheorem{problem}[equation]{Problem}
\newtheorem{main}{Theorem}
\newtheorem*{main*}{Theorem}
\newtheorem*{conjecture}{Conjecture}
\newtheorem*{cor*}{Corollary}

\theoremstyle{fancy2}
\newtheorem{definition}[equation]{Definition}
\newtheorem{rem}[equation]{Remark}
\newtheorem*{rem*}{Remark}
\newtheorem{example}[equation]{Example}

\newcommand{\cref}[1]{Corollary~\ref{#1}}

\newcommand{\Sph}{\mathbb{S}}
\newcommand{\Disc}{\mathbb{D}}
\newcommand{\W}{\mathsf{W}}

\newcommand{\N}{\mathsf{N}}

\newcommand{\RP}{\mathbb{R\mkern1mu P}}
\newcommand{\CP}{\mathbb{C\mkern1mu P}}
\newcommand{\HP}{\mathbb{H\mkern1mu P}}

\newcommand{\C}{{\mathbb{C}}}
\newcommand{\R}{{\mathbb{R}}}
\newcommand{\Z}{{\mathbb{Z}}}

\renewcommand{\H}{\ensuremath{\operatorname{\mathsf{H}}}}
\newcommand{\E}{\ensuremath{\operatorname{\mathsf{E}}}}
\newcommand{\F}{\ensuremath{\operatorname{\mathsf{F}}}}
\newcommand{\G}{\ensuremath{\operatorname{\mathsf{G}}}}
\newcommand{\D}{\ensuremath{\operatorname{\mathsf{D}}}}
\newcommand{\SO}{\ensuremath{\operatorname{\mathsf{SO}}}}
\renewcommand{\O}{\ensuremath{\operatorname{\mathsf{O}}}}
\newcommand{\Sp}{\ensuremath{\operatorname{\mathsf{Sp}}}}
\newcommand{\U}{\ensuremath{\operatorname{\mathsf{U}}}}
\newcommand{\SU}{\ensuremath{\operatorname{\mathsf{SU}}}}

\newcommand{\Spin}{\ensuremath{\operatorname{\mathsf{Spin}}}}

\newcommand{\T}{\ensuremath{\operatorname{\mathsf{T}}}}
\renewcommand{\S}{\ensuremath{\operatorname{\mathsf{S}}}}

\newcommand{\M}{\ensuremath{\operatorname{\mathsf{M}}}}
\newcommand{\A}{\ensuremath{\operatorname{\mathsf{A}}}}
\newcommand{\CC}{\ensuremath{\operatorname{\mathsf{C}}}}
\newcommand{\B}{\ensuremath{\operatorname{\mathsf{B}}}}

\renewcommand{\r}{\ensuremath{\operatorname{\mathsf{r}}}}
\renewcommand{\a}{\ensuremath{\operatorname{\mathsf{a}}}}
\newcommand{\w}{\ensuremath{\operatorname{\mathsf{w}}}}
\renewcommand{\u}{\ensuremath{\operatorname{\mathsf{u}}}}
\renewcommand{\v}{\ensuremath{\operatorname{\mathsf{v}}}}
\renewcommand{\1}{\ensuremath{\operatorname{\mathsf{1}}}}
\newcommand{\g}{\ensuremath{\operatorname{\mathsf{g}}}}
\newcommand{\h}{\ensuremath{\operatorname{\mathsf{h}}}}

\def\con#1=#2(#3){#1 \equiv #2 \bmod{#3}}

\newcommand{\rank}{\ensuremath{\operatorname{rk}}}

 \DeclareMathOperator{\Fix}{Fix}

\newcommand{\no}{\noindent}

\newcommand{\K}{\mathsf{K}}
\renewcommand{\L}{\mathsf{L}}

\renewcommand{\F}{\mathsf{F}}

\begin{document}

\date{\today}

\title{Tits Geometry and Positive Curvature}

\author{Fuquan Fang}
\address{Department of Mathematics\\
Capital Normal University\\
Beijing 100048
\\China}
\email{fuquan\_fang@yahoo.com }

\author{Karsten Grove}
\address{Department of Mathematics\\
University of Notre Dame\\
      Notre Dame, IN 46556\\USA
      }
\email{kgrove2@nd.edu}

\author{Gudlaugur Thorbergsson}
\address{Mathematisches Institut\\
Universit\"at zu K\"oln\\
Weyertal 86-90\\
50931 K\"oln\\ Germany
      }
\email{gthorber@mi.uni-koeln.de }

\thanks{The first author is supported in part by an NSFC grant and he is grateful to the University of Notre Dame for its hospitality.
The second author is supported in part by an NSF grant, a research chair at the Hausdorff Center at the University of Bonn and by a Humboldt research award. The third author is grateful to the University
of Notre Dame and the Capital Normal University in Beijing for their
hospitality. }

\begin{abstract}
There is a well known link between (maximal)  polar representations
and isotropy representations of  symmetric spaces provided by Dadok.
Moreover, the theory by Tits and Burns-Spatzier provides a link
between irreducible symmetric spaces of noncompact type of rank at
least three and irreducible topological spherical buildings of rank
at least three.

We discover and exploit a rich structure of a (connected) chamber system  of finite (Coxeter) type $\M$ associated with any polar action of
 cohomogeneity at least two on any simply connected closed positively curved manifold. Although this chamber system is typically not a Tits
 geometry of type $\M$, its universal Tits cover indeed is a building in all but two exceptional cases.  We construct a topology on this
  universal cover making it into a compact spherical building in the sense of Burns and Spatzier. Using this structure we classify up to equivariant diffeomorphism all polar
  actions on (simply connected) positively curved manifolds of cohomogeneity at least two.
\end{abstract}

\maketitle

The interest in positively curved manifolds goes back to the
beginning of Riemannian geometry or even to spherical and projective
geometry. Likewise, the program of Tits to provide an axiomatic
description of geometries whose automorphism group is a noncompact
simple algebraic or Lie group  goes back to projective geometry.

The presence of symmetries has played a significant role in the
study of positively curved manifolds during the past two decades;
see, e.g.,~the surveys \cite{Gr, Wi1, Zi}. Not only has this resulted
in a number of classification type theorems, it has also lead to new
insights about structural properties, see, e.g.,~\cite{VZ, Wi3}, as
well as to the discovery and construction of a new example
\cite{De,GVZ}.

Unlike \cite{GWZ}, our
work here is not motivated by the quest for new examples. On
the contrary, we wish to explore \emph{rigidity properties} of
special actions on positively curved manifolds whose \emph{linear
counterparts} by work of Dadok
 \cite{Da}, Cartan (see \cite{He}), Tits \cite{Ti1}, and Burns and Spatzier \cite{BSp} ultimately are
 described axiomatically via so-called \emph{compact spherical buildings}.

The special actions we investigate are the so-called \emph{polar
actions}, i.e., isometric actions for which there is an (immersed)
submanifold, a so-called \emph{section}, that meets all orbits
orthogonally.  Such actions form a particularly simple, yet very
rich and interesting class of manifolds and actions closely related
to the transformation group itself. The concept goes back to
isotropy representations of symmetric spaces. Also,  as a special
case, the adjoint action of a compact Lie group on itself is polar
with section a maximal torus. Its extension to general manifolds was
pioneered by  Szenthe in \cite{Sz} and independently by Palais and
Terng in  \cite{PTe},  and has recently been  further developed
 in \cite{GZ}.
 Since the action by the identity component of a polar action is itself polar, we \emph{assume throughout} without further comments that \emph{our
  group is connected}. An exceptional but important special case is that of cohomogeneity one actions and manifolds, i.e., actions with 1-dimensional orbit space.

The \emph{exceptional case} of positively curved cohomogeneity one
manifolds was studied in \cite{GWZ} and \cite{Ve}. Aside from the rank one
symmetric spaces, this also includes infinite families of other
manifolds, most of which are not homogeneous even up to homotopy. In
contrast, our main result here is the following:

\begin{main} \label{A}
A polar action on a simply connected, compact, positively curved
manifold of cohomogeneity at least two is equivariantly
diffeomorphic to a polar action on a compact  rank one symmetric
space.
\end{main}

This is reminiscent of the situation for
\emph{isoparametric} submanifolds in euclidean spheres, where many
isoparametric hypersurfaces are not homogeneous  (see
\cite{OT,FKM}), whereas in higher codimensions by \cite{Th} they are
the orbits of linear polar actions  if they are irreducible or
equivalently
 the orbits of isotropy representations of compact symmetric spaces by \cite{Da}.

All polar actions on the simply connected,  compact rank one
symmetric spaces, i.e., the spheres and projective spaces, $\Sph^n,
\CP^n, \HP^n$ and $\Bbb{OP}^2$ were classified in \cite{Da} (see also \cite{EH}) and
\cite{PTh, GK}. In all cases but $\Bbb{OP}^2$ they are either linear
polar actions on a sphere or they descend from such actions to a
projective space. By the work mentioned above by Dadok, Cartan,
Tits, and Burns-Spatzier, the (maximal) irreducible polar  linear
actions are in 1-1 correspondence with irreducible \emph{compact
 spherical buildings}. On $\Bbb{OP}^2$ any polar action has
either cohomogeneity one or two, and in the second case all but two
have a fixed point. The latter are actions by $\SU(3)\SU(3)$ \cite{PTh} and $\SO(3)\G_2$ \cite{GK}  both with
orbit space a spherical triangle with angles $\pi/2$, $\pi/3$ and
$\pi/4$. We refer to these as the
\emph{exceptional} (irreducible) actions on $\Bbb{OP}^2$.

We note that Theorem A is optimal: In fact, since the Berger-Cheeger deformation \cite{Ch} preserves polarity and lower curvature bonds there are even invariant positively curved polar metrics on any rank one symmetric space arbitrarily Gromov - Hausdorff close to its orbit space.\\

There are different steps and strategies involved in the proof of Theorem A. To guide the reader we provide a short discussion of the key results needed in the proof.

Our point of departure is the following description of sections and
their (effective) stabilizer groups referred to as  \emph{polar
groups} in \cite{GZ} and generalized Weyl groups in \cite{Sz,PTe}:

\begin{main}
The polar group of a simply connected positively curved polar
manifold of cohomogeneity at least two is a Coxeter group or a $\Bbb
Z_2$ quotient thereof. Moreover, the section with this action is
equivariantly diffeomorphic to a sphere, respectively a real
projective space with a linear action.
\end{main}

If this linear action is irreducible we say that the polar $\G$ action is \emph{irreducible}, and reducible otherwise. The above result also allows us to associate a (connected) \emph{chamber system}  $\mathscr{C}(M;\G)$ (cf. \cite{Ti2, Ro}) of type $\M$ (the Coxeter matrix of the
associated Coxeter group), to any simply connected positively curved  polar $\G$ manifold $M$ of cohomogeneity at least two. We point out that in this
generality, the geometric realization of  $\mathscr{C}(M;\G)$ is not always a simplicial complex, so not a \emph{geometry} of type $\M$ in the sense of Tits. For this we prove:

\begin{main}
Let $M$ be a simply connected positively curved polar $\G$ manifold
without fixed points, and not (equivalent to) an exceptional action
on $\Bbb{OP}^2$. Then the universal cover
$\tilde{\mathscr{C}}(M;\G)$ of $\mathscr{C}(M;\G)$ is a spherical
building.
\end{main}

Moreover, the Hausdorff topology on compact subsets of $M$ induces in a natural way a topology on $\tilde{\mathscr{C}}(M;\G)$ for which we prove:

\begin{main}
Whenever the universal cover $\tilde{\mathscr{C}}(M;\G)$  of
$\mathscr{C}(M;\G)$ is a building, it is a compact spherical
building.
\end{main}

When the Coxeter diagram for $\M$ is connected, or more generally
has no isolated nodes, the work of Burns and Spatzier \cite{BSp} as
extended by Grundh\"{o}fer, Kramer, Van Maldeghem and Weiss
\cite{GKMW} applies, and hence $\tilde{\mathscr{C}}(M;\G)$  is
the building of the sphere at infinity of a noncompact symmetric
space $\U/\K$ of nonpositive curvature, and the action of $\K$ on
the sphere at infinity is the linear polar action whose chamber
system is the building. In our case, the fundamental  group $\pi$ of
the cover becomes a compact normal subgroup of $\tilde \G \subset
\K$ acting freely on the sphere with quotient our manifold with the
action by $\G = \tilde \G /\pi$. Moreover, the actions by $\tilde \G
$ and $\K$ on the sphere are orbit equivalent. This already proves
our Theorem \ref{A} up to equivariant homeomorphism in this case (Theorem \ref{mainresult}),
and equivariant diffeomorphism follows, e.g., from the \emph{recognition
theorem} in \cite{GZ}. In particular, we note that in this case $M$ is either a sphere or a quotient thereof by a Hopf action, i.e., not the Cayley plane.


In the remaining (reducible) cases (including the case of fixed
points), where isolated nodes of the Coxeter diagram are present,
the above mentioned extended Burns-Spatzier-Tits theory  does not
yield the desired result, and we also use more direct geometric
arguments that hinges on a \emph{characterization of Hopf fibrations} in our
context, Lemma \ref{hopf}.


We point out that the proof of Theorem C  above has  three distinct
parts, a special one of which is carried out in \cite{FGT}. For all chamber systems of rank at least four, our
constructions combined with the work of Tits gives the result (cf. Theorem \ref{building}). In the reducible rank three cases it follows from Theorem \ref{nonexcep}. In
the irreducible rank three cases, i.e., of type $\A_3$ and $\CC_3$, corresponding to the orbit space
being a spherical triangle with angles $\pi/2$, $\pi/3$ and $\pi/3$, respectively $\pi/2$, $\pi/3$ and $\pi/4$,
the general theory breaks down. The point of departure here (Theorem \ref{simplicial}), is that in
this case $\mathscr{C}(M;\G)$ is \emph{simplicial}, thus an $\A_3$, respectively a $\CC_3$
\emph{geometry}. Since all $\A_n$ geometries are buildings by work of Tits, this completes the case of $\A_3$. In the case of $\CC_3$ one can use an axiomatic characterization of
buildings of type $\CC_3$ due to Tits. This is carried out in \cite{FGT} via reductions and the work in \cite{GWZ}. Unlike the higher rank case,  the strategy here is to \emph{construct suitable covers} and prove that they are buildings. Exactly two cases emerge (from the exceptional actions on $\Bbb{OP}^2$) where this cannot be done due to theorem D and the subsequent discussion above. In fact,  we
conclude:

\begin{main}
The chamber system $\mathscr{C}(\Bbb{OP}^2, \G)$  for an irreducible polar $\G$ action
on $\Bbb{OP}^2$ is a $\CC_3$ geometry whose universal cover is not a building.
\end{main}

 The existence of  $\CC_3$ geometries whose universal covers are not buildings are well known in the
``real estate community" (see \cite{Ne}), but the
examples $\mathscr{C}(\Bbb{OP}^2, \G)$, with $\G = \SU(3)\cdot \SU(3)$, and $\G = \SO(3) \cdot \G_2$ which arise naturally in our context are new.

\medskip
We conclude this outline by pointing out that positive curvature is used in the general theory for two purposes:
(1) To prove Theorem B (cf. Section 2) , and (2) to establish that the associated chamber system $\mathscr{C}(M;\G)$ is connected, or equivalently the $\G$ action is \emph{primitive} (cf. Section 3). In fact, we prove the conclusion of Theorem A for any polar action with connected chamber system and whose orbit space has positive curvature, unless it is of type  $\CC_3$. In the case of $\CC_3$ (see \cite{FGT}) we use positive curvature more extensively, in particularly relying on the work in \cite{GWZ} (alternatively, due to Theorem \ref{simplicial}, this case is also covered by the classification of  irreducible homogeneous geometries of finite Coxeter type $\M$ of rank at least two in \cite{KL}).

\smallskip
Unlike previous applications of buildings to geometry, what is essential for us is to use Tits's \emph{local approach to
 buildings}, i.e., via chamber systems and their universal covers \cite{Ti2, Ro}.  We like to mention that this is the case also in independent simultaneous work by Lytchak \cite{Ly} describing the structure of polar singular foliations of codimension at least
three in compact symmetric spaces. In particular, in his context results similar to Theorem D and its corollary, Theorem E were obtained.

\bigskip

We point out that a corresponding theory for polar actions on
\emph{nonnegatively curved manifolds} is significantly more involved. In particular, the concept of an ``irreducible action" is not as
straightforward in this case, in part since the section is no longer just a
sphere or a real projective space with a reflection group (for a complete description see \cite{FG}). For example the polar $\T^2$ action on $\CP^2$ with three
fixed points induces a polar $\T^2$ action on $\CP^2 \# \pm \CP^2$
with a metric of nonnegative curvature and flat Klein bottle as
section (see \cite{GZ}) that should be viewed as reducible, as should actions that are not primitive.

With the
appropriate notion of irreducibility we

\begin{conjecture}
An irreducible polar action on a simply connected nonnegatively
curved compact manifold is equivariantly diffeomorphic to a quotient of a polar
action on a symmetric space.
\end{conjecture}

\smallskip

Another interesting direction is based on part of
our work here that generalizes to curvature free settings. For
example, the combinatorial content of our paper can be used in the study of general polar actions on simply connected manifolds with finite polar group.

\bigskip

We have divided the paper into seven sections. Structurally it consists of three rather different moderately intertwined  parts, Sections 1-4 constituting Part I, Section \ref{sectionC3gen} (together with \cite{FGT}) Part II, and Sections \ref{input}-\ref{genRed}  Part III. Here Part I deals with the overall general approach and theory leading to a proof of Theorem A for all irreducible actions of cohomogeneity at least three. Part II deals with the exceptional case of irreducibe actions of cohomogeneity two where the key issue for the general theory breaks down. Finally Part III deals with all reducible cases including cohomogeneity two. In particular, Parts I and III yield a proof of Theorem A in cohomogeneity at least three. In cohomogeneity two only the irreducible actions of type $\CC3$ are not covered in this paper and we refer to \cite{FGT} (or \cite{KL}).

The first two sections are devoted to preliminaries and an analysis of sections culminating in Theorem B, which actually provides a complete classification of positively curved manifolds with reflection groups.
The chamber system associated with a polar action in positive curvature is investigated in Section 3. The point of departure here is  that this chamber system is connected.
The proof of this is based on a
result about dual foliations due to Wilking \cite{Wi3}.
  We conclude Section 3 by proving Theorem C in all
(irreducible) cases but $\A_3$ and $\CC_3$.

In Section 4 we equip the ingredients of Theorem C with a natural topology based on the classical Hausdorff topology on closed sets in a compact metric space.
 Our main result here is
that with this topology the universal covers of our chamber systems
are compact spherical buildings in the sense of Burns and Spatzier.
This then in particular leads to a proof of Theorem A for
 all irreducible actions but those of type $\A_3$ and $\CC_3$.

As mentioned above, the general theory for compact spherical buildings breaks down for
reducible actions in general (the ones for whom the Coxeter diagram
has isolated nodes). The proof of Theorem A for such actions is
carried out in Sections \ref{input}  and \ref{genRed}. As a key input, we provide in Section \ref{input}
 a characterization of Hopf fibrations in our context which is of
independent  interest.  This immediately yields Theorem A for the
special case where fixed points are present. The reducible case
where no fixed points are present is dealt with in Section \ref{genRed}.

For basic facts and tools involving critical point theory for nonsmooth distance functions and convex sets in positive curvature that will be used freely we refer to \cite{Pe} Chapter 11.

It is our pleasure  to thank Linus Kramer, and Alexander Lytchak for constructive discussions and comments. Likewise, we are grateful to an anonymous referee for constructive comments and a suggestion that lead to a significant simplification of the proof of the Hopf Lemma \ref{hopf}, and subsequent ramifications.

\section{Preliminaries}    \label{sectionone}
We will begin by giving a brief description of known facts for
general polar manifolds (cf. e.g. \cite{GZ} and \cite{HPTT} for
further information). We observe that under
fairly mild restrictions, there is a general so-called chamber system naturally
associated with such actions. We will end the section with a
description of such systems, and the special case  of Coxeter
systems.

\bigskip

Throughout $\G$ will be a compact connected Lie group acting
isometrically on a  connected compact Riemannian manifold $M$  in a
polar fashion. By definition there is a  \emph{section} $\Sigma$, i.e., an
immersion $\sigma: \Sigma \to M$ of a connected manifold $\Sigma$,
whose image intersects all $\G$ orbits orthogonally. Moreover, we
demand that $\sigma$ is a section without a subcover section, i.e. $\sigma$ does not factor through a covering $\Sigma \to \Sigma '  \to M$.
Obviously $\g\sigma$ is a section for any $\g \in \G$, and
$\G\sigma(\Sigma) = M$. Clearly, $\Sigma$ has the same dimension as
the \emph{orbit space}  $M^* := M/\G$, i.e., the
\emph{cohomogeneity} of the action, or the codimension of
\emph{principal orbits}, $\G/\H \subset M$. If not otherwise stated,
it is understood that $0 < \dim M^* < \dim M$. This eliminates
general actions by discrete groups, and general transitive actions.
In addition, we also assume that $M$ is not a product where $\G$
acts trivially on one of the factors. In general, we will denote the
image of a subset $X \subset M$ under the \emph{orbit map} by $X^*
\subset M^*$.

\bigskip

The following facts are simple and well known (cf. \cite{Sz,PTe}):

\begin{itemize}

\item
Any section is \emph{totally geodesic}.

\item
The slice representation of any isotropy group $\K \subset \G$ is a
\emph{polar representation}.
\end{itemize}

\no Recall here that if $\K$ fixes $p \in M$, then the \emph{slice representation} of $\K$ is the action of $\K$ by differentials on the normal space in $T_pM$ to the tangent space $T_p(\G p)$ of its orbit $\G p$. We often restrict this further to the subspace $T_p^{\perp}$ perpendicular also to the fixed point set $T_pM^{\K}$. This is also a polar representation.

 Fix a section $\sigma: \Sigma \to M$ and a point $p \in \Sigma$
corresponding to a principal $\G$ orbit, i.e., $\G \sigma(p)$ is a
principal orbit with isotropy group $\H = \G_{\sigma(p)}$. The
stabilizer subgroup $\G_{\sigma(\Sigma)}  \subset \G$ of
$\sigma(\Sigma)$ induces an action on $\Sigma$. Clearly, $\H$ is the
kernel of  that action, and we refer to $\Pi: = \G_{\sigma(\Sigma)}/
\H$ as the \emph{polar group} associated to the section $\sigma$.
Recall the following facts:

\smallskip

\begin{itemize}

\item
For any $q \in \Sigma$, $\sigma_{*} (T_q \Sigma) \subset
T_{\sigma(q)}M$ is a section of the polar representation, the slice
representation of $\G_{\sigma(q)}$, and the associated polar group
is the isotropy group $\Pi_q$.

\item
 $\Pi$  is a discrete subgroup of $\N(\H)/\H$ acting properly discontinuously on $\Sigma$  with trivial principal isotropy group.

\item
$M^* = \Sigma^* := \Sigma/\Pi$ is an orbifold.
\end{itemize}

\medskip

In complete generality, the structure of $M$ and its $\G$ action is
encoded in the section $\Sigma$, the polar group $\Pi$ and its
actions on $\Sigma$ and $\G/\H$, and the $\G$ isotropy groups along
$\Sigma$. Although in general, $\Pi$ can be any
discrete subgroup of a Lie group, typically singular orbits are present, in which case there is
a
nontrivial normal subgroup
 $\W \subset \Pi$ generated by \emph{reflections} $\r_i$ associated with maximal singular isotropy groups
 $\K_i  \subset \G$ along $\Sigma$  (cf. \cite{GZ}). We refer to any group generated by reflections as a \emph{reflection group}. We stress that here $\r: \Sigma \to \Sigma$ is called a reflection if $\r$ has order two, and at least
  one component of the fixed point set has codimension 1. The codimension 1 components $\Lambda_{\r} \subset \Sigma$ of the fixed point set $\Sigma^{\r}$ are referred to as the
   \emph{mirrors} of $\r$. A connected component $c$ of the complement of all mirrors is called
   an (\emph{open}) \emph{chamber} of $\Sigma$. We denote the closure of an open chamber by $C = \bar c$ and refer to it simply as a \emph{chamber}. Again, we
   stress that this kind of terminology is usually reserved to the situation where the complement of a mirror has two connected components interchanged by the reflection. Note that the latter is automatic for reflections on a simply connected manifold.

It is clear that $\W$ acts transitively on the set of open chambers
of $\Sigma$, but the stabilizer group $\W_c$ which we will call the
\emph{chamber group} may be nontrivial  when the section is not
simply connected (cf. Example \ref{trivial examples}
and Theorem \ref{1-connected}). Clearly $\Sigma/\W = C/ \W_c$. Moreover, the boundary $\partial C = C - c$ of a chamber $C$, is
the union of its  \emph{chamber faces}, where a chamber face is a
nonempty intersection $C \cap \Lambda$ with a mirror.

The following examples illustrate these concepts and are relevant
for our subsequent discussion about positive curvature:

\begin{example}\label{trivial examples}
Consider the following groups $\W$ acting on $\Sph^2$ as well as on
$\RP^2$.

\no (1) $\W = \A_1 = \langle \r\rangle $, where $\r$ is the
reflection in the equator:

On $\Sph^2$ there is one mirror and two open chambers, the open
upper and lower hemispheres interchanged by $\r$. Their closure is
the orbit space $\Sph^2/\W$. There is one face, its boundary circle
(and it coincides with the mirror of $\r$).

On $\RP^2$ there is one mirror and one open chamber and it is
preserved by $\r$. Its closure is all of $\RP^2$ and the orbit space
$\RP^2/\W$ is the cone on its boundary circle, the cone point
corresponds to the isolated fixed point of $\r$ on $\RP^2$ (in the
chamber). There is one face, the whole boundary (and it coincides
with the mirror of $\r$).

Note that the action of $\W$ on $\RP^2$ lifts to the action of $\W$
on $\Sph^2$. If we extend this action by $-{\rm id}$, the extended
group action induces the same action on the base, and now has the
same orbit space.

\no (2) $\W = \A_1 \times \A_1 = \langle \r_0, \r_2\rangle$ where
$\r_0, \r_2$ are reflection in two great circles making an angle
$\pi/2$:

On $\Sph^2$ there are two mirrors and four open chambers. Their
closure is the orbit space $\Sph^2/\W$, a spherical right angled
biangle. There are two faces each of which are also a chamber face.
Their intersection is the intersection of mirrors and coincides with
the fixed point set $\Fix(\W)$.

On $\RP^2$ there are three mirrors and four open chambers. In fact,
the ``rotation" $\r_0\r_2$ on $\Sph^2$ induces a reflection on
$\RP^2$. The closure of an open chamber is the orbit space
$\RP^2/\W$, a right angled spherical triangle. There are three
faces, each of which is also a chamber face. The intersection of all
mirrors is empty, but each vertex of the orbit space triangle
correspond in this case to a fixed point of $\W$.

In this case, the lifted action of $\W$ on $\RP^2$ to $\Sph^2$
contains a rotation of angle $\pi$. Again the extended action by
$-{\rm id}$ defines the same action on $\RP^2$, but on $\Sph^2$ the
action has three reflections, and of course the same orbit space. In
other words, the reflection group on $\Sph^2$ generated by the lift
of all the reflections in $\RP^2$ contains the antipodal map in this
case as opposed to the first case.

\no (3) $\W = \A_2 = \langle\r_0, \r_3\rangle$ where $\r_0, \r_3$
are reflection in two great circles making an angle $\pi/3$:

On $\Sph^2$ there are three mirrors and six open chambers. Their
closure is the orbit space $\Sph^2/\W$, a spherical biangle with
angle $\pi/3$. There are two faces each of which are also a chamber
face. Their intersection is the intersection of mirrors and
coincides with the fixed point set $\Fix(\W)$.

On $\RP^2$ there are three mirrors and three open chambers. The
closure of an open chamber is a spherical biangle with angle $\pi/3$
where the two vertices have been identified! The stabilizer $\W_c$
of a chamber has order two, fixes the ``mid point" of $C$ and
rotates $C$ to itself, mapping one chamber face to the other. The
orbit space has one face with one singular point, the fixed point of
$\W$ and one interior singular point, the fixed point of $\W_c$.

In this case, the reflection group obtained by lifting the
reflections in $\RP^2$ to reflections in $\Sph^2$ does not contain
the antipodal map. If we extend it by the antipodal map we get the
same action on $\RP^2$ and the orbit spaces are of course the same
as well.

\no (4) Consider a linear cohomogeneity one
action on a sphere with orbit space of length $\pi/i$, $i=2,3$.
The suspended action and its induced action on the real projective
space have sections and polar groups as presented in (2) and (3) above.

\end{example}

\smallskip

Note that \emph{if $\Pi = \W$ and the chamber group $\W_c$ is trivial}, it follows
that $C$ is isometrically identified with $\Sigma/\W = M/\G$, and that $\W$ acts simply transitively on the
set of closed chambers of a fixed section $\Sigma$.  Moreover, $\G$
acts transitively on the set of all chambers in all sections of $M$,
i.e. $M =  \cup_{\g \in \G} gC$, and this set of chambers is $\G/\H$
as a set. The chamber faces $F_i, i = 1, \ldots k$, of $C$
correspond to a set of generators $r_i$ for $\W$. This way all faces
of chambers $gC$, $g\in \G$ of $M$ get labeled consistently,
so that $\G$ is label preserving. Now define two chambers $g_1C$ and $g_2C$ to be $i$-\emph{adjacent} if they have a
common $i$ face $g_1F_i = g_2F_i$.  This relation among the set of chambers  in $M$, respectively all chambers in a fixed
section $\Sigma$  make both of these
sets into a \emph{chamber system} $\mathscr{C}(M, \G)$,
respectively $\mathscr{C}(\Sigma, \W)$ according to the following
definition (see, e.g., \cite{Ti2,Ro}):
\smallskip

 An \emph{(abstract) chamber system over $I =\{1,\dots,k\}$} is a set $\mathscr{C}$ together with a partition
 of $\mathscr{C}$ for every $i\in I$. Elements $C, C' \in \mathscr{C}$ in the same part of the $i$-partition, are said to be $i$-\emph{adjacent}
 which is written as $C\sim_i C'$. The elements of $\mathscr{C}$ are called \emph{chambers}.

We will use the following standard terminology in subsequent
sections:

A \emph{gallery}  in $\mathscr{C}$ is a sequence $\Gamma =
(C_0,\dots,C_m)$ in $\mathscr{C}$ such that
 $C_j$ is $i_j$-adjacent to $C_{j+1}$ for every $0\le j\le m-1$. Here the \emph{word} $f = i_0i_2\dots i_{m-1}$ in $I$ is referred to as the \emph{type} of the gallery.
 If we want to indicate this type, we write $\Gamma_f$ rather than just $\Gamma$.  If the $i_j$'s belong to a subset
 $J$ of $I$, we call $\Gamma = (C_0,\dots,C_m)$ a $J$-gallery. A subset $\mathcal{B}$ of a  chamber system $\mathscr{C}$ is said to be
 \emph{connected} (or $J$-\emph{connected}) if any two chambers in it can be connected by
 a gallery (or a $J$-gallery). The $J$-connected components of $\mathscr{C}$ are called $J$-\emph{residues}.
 The \emph{rank} of a $J$-residue is the cardinality of $J$ and its \emph{corank} is the cardinality of
 $I\setminus J$.  Given residues $R$ and $S$ of types $J$ and $K$ respectively, we say that
 $S$ is a \emph{face} of $R$, if $R\subset S$ and $J\subset K$.

\bigskip

Note that for chamber systems  $\mathscr{C}(M, \G)$ as above, if two
mirrors $\Lambda_i$ and $\Lambda_j$ in $\Sigma$ corresponding to two
reflections $\r_i$ and $\r_j$ on $\Sigma$ intersect, then $(\r_i
\r_j)^{m_{ij}} =1$ for some finite integer $m_{ij} > 1$. In fact,
$\r_i, \r_j \in \W_p$ the reflection group of the polar
representation of the isotropy group $\G_p$ for $p \in  \Lambda_i
\cap\Lambda_j =: \Lambda_{ij}$, so  $ \langle\r_i,\r_j \rangle$ is a
dihedral group, and the angle between $\Lambda_i$ and $\Lambda_j$ is
$\pi/m_{ij}$. In fact, in our case $m_{ij}$ is limited to $2$, $3$, $4$, or $6$ , since these are the possibilities for isotropy representations of symmetric spaces, and moreover no exceptional orbits are present (Theorem \ref{ata}).

\medskip

Recall, that a symmetric $k\times k$ matrix $\M=(m_{ij})$ with
entries from $\mathbb{N}\cup \{\infty\}$, with $m_{ii}=1$ for all
$i\in I$, and $m_{ij}>1$ if $i\ne j$ is called a \emph{Coxeter
matrix}.

Pictorially, $\M$ is given by its so-called \emph{diagram}, which
consists of one \emph{node} for each $i \in I$ and $m_{ij} - 2$
lines joining the $i$ and $j$ nodes.

The associated \emph{Coxeter group of type $\M$} is the group
$\W(\M)$ given by generators and relations as
$$
 \W(\M)=\langle\{r_1,\dots,r_k\}\mid (r_ir_j)^{m_{ij}}=1 \hbox{ for all } i,j\in I \hbox{ such that $m_{ij}$ is finite}\rangle.
 $$
 The pair $(\W(\M),S)$, where $S=\{r_1,\dots,r_k\}$, is called the \emph{Coxeter system of type $\M$}, and $k$ is referred to as its \emph{rank}. The elements of $\W(\M)$ that are conjugate to elements
 in $S$ are called \emph{reflections}.

There is a natural chamber system, $\mathscr{C}(\W)$ associated with
a Coxeter system $(\W(\M),S)$, where $S=\{r_1,\dots,r_k\}$ and
$I=\{1,\dots,k\}$: One defines $i$-adjacency for $i\in I$ to be
$w\sim_iwr_i$, i.e.,
 each part in the $i$-partition of $\W$ consists of two elements. Notice that $\W$ is
 connected since $S$ generates $\W$. There is a partial order among residues defined by
 setting $S \le R$ if $S\supset R$. The residues $T$ for which $S \le T$ implies $S=T$ are called the \emph{vertices} of $\mathscr{C}(\W)$. Denote the set of vertices by $\mathcal{V}$. One associates to a residue $S$ the subset $S' \subset \mathcal{V}$ defined by $S' = \{v\in \mathcal{V}| v \le S\}$, and call $S'$ an $i$-\emph{simplex} if its cardinality is $i+1$. The set simplices in $\mathcal{V}$ is denoted by $\Delta(\W)$.

 \smallskip

 The Coxeter complex,  $\Delta(\W)$ associated to a Coxeter system $(\W(\M),S)$ also provides an example of an (abstract) simplicial complex:

 \no Recall that an  \emph{(abstract) simplicial complex} is a nonempty family $\mathcal{S}$ of finite subsets (called \emph{simplices}) of a  set $V$
so that $\{v\}\in \mathcal{S}$ for every $v\in V$ and every subset
of a simplex in $\mathcal{S}$ is a simplex in $\mathcal{S}$ (called
a \emph{face}). (The simplices consisting of one element are called
\emph{vertices}.)

\smallskip

One has the following general facts (see e.g. \cite{Dav} page 179, Theorem 10.1.5 and Lemma 10.1.6):

 \begin{thm}\label{1-connected} For any reflection group $\W$ acting on a simply connected Riemannian manifold $\Sigma$, $(\W,S)$ is a Coxeter system, where $S$ are reflections in $\W$ corresponding to the faces of a chamber $C$, and $\W_c$ is trivial. If $\Sigma$ is compact, $\W$ is finite and isomorphic to a spherical reflection group.
\end{thm}

\begin{rem} \label{finite cox}The Coxeter groups that we will deal with in this paper will all be finite. By a theorem of Coxeter,
 Coxeter systems of rank $k+1$ are in one to one correspondence with finite
subgroups of $\O(k+1)$ that are generated by reflections in
hyperplanes of $\R^{k+1}$
 and only fix the origin.
  Such groups have been classified. Let $(\W,S)$ be a Coxeter system of rank $k$ acting as
   a reflection group on $\R^{k+1}$, and consider its restriction to  $\Sph^{k}$. In this case mirrors are of course great spheres $\Sph^{k-1}$, and the Coxeter group $\W$
   acts simply transitively on the set of chambers. Each chamber is a spherical $k$-simplex and the corresponding triangulation of $\Sph^{k}$ is the \emph{geometric realization} of the Coxeter complex $\Delta(\W)$ associated to a Coxeter system $(\W,S)$.

 The geometry of this representation is also reflected in the Coxeter diagram of $\M$. For example, this diagram  is connected if and only if this action is irreducible.
 Each node corresponds to a codimension one face simplex, $i$, and $\pi/m_{ij}$ is the angle between the corresponding $i$ and $j$ faces of  the $k$-simplex $\Sph^k/\W$. The Coxeter diagram for the isotropy  group of $\W$ at the vertex opposite of face $i$ is obtained from the Coxeter diagram of $\W$ by removing the $i$-th node.
   \end{rem}

  We note that the chambers for the Coxeter system $(\W,S)$ in Theorem \ref{1-connected} when
  $\Sigma$ is a compact $k$ manifold combinatorially are the
  same as the spherical $k$ simplices of its representation above. Geometrically, it follows in particular that all angles in a chamber of $\Sigma$ are the same as the corresponding angles in the spherical simplex.

  Although $\A_2$ is an irreducible Coxeter group, we point out that all the linear 3-dimensional representations presented in Example \ref{trivial examples} above are reducible.
  We conclude this section with important examples of irreducible
Coxeter groups:

 \begin{example}  Finite Coxeter groups that are isomorphic to finite and irreducible reflection groups acting on $\mathbb{R}^3$ will play a special role in some of our proofs.
 There are three such
 groups that in the classification of finite Coxeter (or reflection) groups are given the symbols
 $\A_3, \CC_3, \H_3$.

 The group $\A_3$ is isomorphic to the symmetric group on four letters. It is the group of symmetries of a regular tetrahedron. Its order is $24$. The $2$-simplexes in the triangulation
 explained
 above have angles $\pi/2$, $\pi/3$, and $\pi/3$ at the vertices.

 The group $\CC_3$ is the symmetry group of a regular cube (or dually of a regular
 octahedron). Its order is $48$. The $2$-simplices in the triangulation
 have angles $\pi/2$, $\pi/3$, and $\pi/4$ at the vertices.

 The group $\H_3$ is the symmetry group of a regular dodecahedron (or dually of a
 regular icosahedron). Its order is $120$. The $2$-simplices in the triangulation have angles $\pi/2$, $\pi/3$, and $\pi/5$ at the vertices.
 Note, that the occurrence of the angle $\pi/5$ \emph{excludes} $\H_3$ as a Coxeter group of a polar action.
 \end{example}

   \section{Sections and Coxeter Groups}  
We assume from now on that $M$
is a \emph{positively curved polar $\G$-manifold of cohomogeneity at least two}. This will yield
strong restrictions on all the basic items presented in Section \ref{sectionone}.
In particular, we will prove that sections are either spheres or
real projective spaces.

When $M$ is simply connected, we show that the polar group is a Coxeter group when
the section is a sphere, and a $\Z_2$ quotient of such a group when
the section is a real projective space; in either case the action is
linear as stated in Theorem B of the introduction.

The starting point is the following

\begin{lem}[Singular Orbit] \label{sing} Any positively curved polar $\G$-manifold has singular orbits.
\end{lem}

\begin{proof} If all orbits have maximal dimension, the normal distribution  is globally defined and integrable with leaves the sections of $M$.
Since in particular the sectional curvature of $M$ is nonnegative,
it now follows from Theorem 1.3 in \cite{Wa} that the orbits of $\G$
are totally geodesic, and that the metric on $M$ locally is a
product metric. This is a contradiction since the sectional
curvature of $M$ is actually positive.
\end{proof}

\begin{rem}
For a nonnegatively curved polar $\G$-manifold the same conclusion
holds unless $M = \Sigma \times_{\Pi} \G/\H$ is locally metrically a
product. If in addition $M$ is simply connected, $M = \Sigma \times
\G/\H$ with a product metric.
\end{rem}

From Section 1, we know in particular that  the reflection group $\W
\subset \Pi$ is nontrivial and that $\partial M^* = \partial
\Sigma^*$ is nonempty. This already is sufficient to prove

\begin{prop}[Section] \label{dif-sec}
Let $M$ be a compact positively curved polar manifold. Then any
section $\Sigma$ is diffeomorphic to either  a sphere $\Sph^k$ or a
real projective space $\RP^k$. In particular, the polar group $\Pi$
is finite.
\end{prop}

\begin{proof}
Let $\r$ be a reflection, with mirror $\Lambda$ and $E \subset
\Lambda$ a component. Since the curvature is positive the (local)
distance function to $E$ is strictly concave. In particular, the
complement $\Sigma - D_{\epsilon}(E)$ of a small tubular
neighborhood of $E$ is a (locally) convex set with boundary
$\partial D_{\epsilon}(E)$. This set either has  one or two
components corresponding to the boundary having one or two
components. In either case, each component is a disc by the standard
``soul argument", and in fact $E = \Lambda$. The key fact here is
that the distance function to the boundary is strictly concave and
hence has a unique point at maximal distance called the \emph{soul
point}. Moreover the distance function to the soul point has no
critical points. For the arguments and constructions below it is
also important that the distance function is $\r$ invariant.

In the case, where $\Sigma - D_{\epsilon}(E)$ has two components,
$\Lambda = E$ separates $\Sigma$ into two manifolds $V_+$ and $V_-$
each with $\Lambda$ as a totally geodesic boundary.  In this case
the isometry $\r$ interchanges $V_+$ and $V_-$. Moreover, the
diffeomorphism $\phi$ say from the upper hemisphere $\Disc^k_+$ of
$\Sph^k$ to $V_+$ can be chosen so that the north pole of
$\Disc^k_+$ goes to the soul point of $V_+$, and the image of the
gradient lines to the north pole of $\Disc^k_+$ are ``radial" near
$E$ and the soul point. The map $\Phi : \Sph^k \to \Sigma$ defined
by $\Phi = \phi$ on $\Disc^k_+$ and $\Phi = \r \phi \rho$ on
$\Disc^k_-$ is a diffeomorphism which is equivariant relative to the
reflections $\rho$ in the equator of $\Sph^k$ and $\r$ on $\Sigma$.

In the case where $\Sigma - D_{\epsilon}(E)$ has one component,
$\r$ fixes its soul point and acts freely elsewhere: In fact $\r$
clearly acts freely in $D_{\epsilon}(E) - E$, so by convexity $\r$
can only have isolated fixed points in  $\Sigma - D_{\epsilon}(E)$.
Moreover, if there was an isolated fixed point in addition to the
soul point a minimal geodesic between them would be reflected to a
closed geodesic which is impossible by convexity. In particular,
$\Lambda = E = \RP^{k-1}$ and $\Sigma$ has fundamental group $\Z_2$.
In the two fold universal cover $\tilde{\Sigma}$ of $\Sigma$, the
lift $\tilde{ \Lambda}$ splits $\tilde{\Sigma}$ into two convex
components $V_+$ and $V_-$ with common totally geodesic boundary
$\tilde{ \Lambda}$, as in the first part. The reflection $\r$ lifts
to a reflection $\tilde{\r}$ interchanging  $V_+$ and $V_-$, each
being mapped isometrically by the projection map to $\Sigma -
\Lambda$ (see also remark \ref{invol}). Choosing a diffeomorphism
$\phi$ say from the upper hemisphere $\Disc^k_+$ of $\Sph^k$ to
$V_+$ as before, the map $\tilde{\Phi}:  \Sph^k \to \tilde{\Sigma}$
defined by $\tilde{\Phi} = \phi$ on $\Disc^k_+$ and $\tilde{\Phi} =
\r \phi \tilde \rho$ on $\Disc^k_-$ is a diffeomorphism which is
equivariant relative to the reflection $\tilde \rho$ on $\Sph^k$ and
$\tilde \r$ on $\tilde{\Sigma}$, and in addition by construction
equivariant relative to the antipodal map $-{\rm id}$ on $\Sph^k$
and the \emph{deck transformation} $\a$ of $\tilde \Sigma$. We conclude
that $\tilde{\Phi}$ induces a diffeomorphism $\Phi: \RP^k \to
\Sigma$ which is equivariant relative to the reflections $\rho$ on
$\RP^k$ induced from $\tilde \rho$ and $\r$ on $\Sigma$.
 \end{proof}

 \begin{rem}\label{invol}
 During the proof of the result above we note in particular that if $\r \in \W$ is a reflection of the section $\Sigma$ with mirror $\Lambda$,
then:
\begin{itemize}
\item
If $\Sigma$ is a sphere, $\Fix(\r) = \Lambda$ and $\Lambda$ is a
codimension one sphere.
 \item
If $\Sigma$ is a projective space, $\Fix(\r) = \Lambda \cup s$,
where $s$ is the soul point at maximal distance to $\Lambda$, and
$\Lambda$ is a real projective space of codimension one. - Note in addition that $\r$  also lifts to a map preserving $V_{\pm} \subset \tilde M$ and
acting as $\a$ on $\tilde \Lambda$
 \end{itemize}

\no In particular, \emph{mirrors are connected}, and if $\Pi =
\langle \r \rangle$ , the result above gives a complete equivariant
description of $(\Sigma, \Pi)$.
\end{rem}

The proof above also allows us to derive further information about
the reflection group $\W$ and the corresponding  open chambers and
orbit space $\Sigma/\W$:

\begin{lem}[Sphere Chamber]\label{w-sphere}
Assume $\Sigma$ is a $k$-dimensional sphere, and $c$ is an open
chamber. Then
\begin{itemize}
\item
Intersections of mirrors are spheres, and the closure $C$  of $c$ is a
convex set in $\Sigma$.
\item
There are at most $k+1$ chamber faces, and the intersection of all
of them is $\Fix(\W)$.
\item
If there are $k+1$ chamber faces, then $C$ is a $k$-simplex, and
$\Fix(\W) = \emptyset$.
\item
If there are  $\ell+1 < k+1$ chamber faces, then $C$ is the join of
$\Fix(\W)$ with an $\ell$-simplex.
\end{itemize}

Moreover, $C$ is a fundamental domain for $\W$ and $\Sigma/\W = C$.
\end{lem}

\begin{proof}
If there is only one mirror $\Lambda$ corresponding to one
reflection $\r$, $\W = \langle\r \rangle=\Z_2$ and $C$ is a closed
convex disc with boundary $\Lambda = \Fix(\W)$ and indeed a join of
$\Fix(\W)$ with a $0$-simplex, the soul point $s$ of  $C$ as we have
seen.

Now consider any two reflections, $\r_i, i=1,2$ with corresponding
mirrors $\Lambda_i$. If $p \in \Lambda_{12}:=\Lambda_1 \cap
\Lambda_2 = \Fix( \langle\r_1,\r_2 \rangle)$, clearly  $\r_i \in
\W_p$ the reflection group of the polar representation of the
isotropy group $\G_p$. In particular, $ \langle\r_1,\r_2 \rangle$ is
a dihedral group, and the angle between $\Lambda_1$ and $\Lambda_2$
is $\pi/k$ for some integer $k$. In particular, the intersection
$\Lambda_{12}$ is a codimension one totally geodesic submanifold of
either mirror $\Lambda_i$, and hence again by convexity is a sphere
(two points when the mirrors are 1-dimensional).

In general, consider $\ell$ mirrors $\Lambda_1, \ldots,
\Lambda_{\ell}$ such that the inclusions of iterated intersections
$\Lambda_{12} \supset \Lambda_{123} \supset \ldots \supset
\Lambda_{123 \ldots \ell}$ are all strict. Then each intersection is
a totally geodesic submanifolds of codimension one in the previous
intersection, and hence $\Lambda_{123 \ldots \ell}$ is a $(k-\ell)$
sphere. Also  $\Lambda_{123 \ldots \ell}$ is the set fixed
by all reflections $\r_i$ with corresponding mirror $\Lambda_i$.
This completes the proof of the first two ``bullets", since mirrors
corresponding to $\ell$ different chamber faces satisfy the needed
inclusion property.

Now suppose $C$ has $\ell+1$ chamber faces, $F_0, \ldots, F_{\ell}$.
Since the angle between any two faces is at most $\pi/2$ it follows as in the original Cheeger-Gromoll case of the distance function to the full boundary \cite{Pe}, that the distance function on $C$ to one face, say $F_0$ is
strictly concave (cf. Theorem 7 in \cite{Wi2}, Theorem 1.3 in \cite{GKi},  and Corollary 3.2 in \cite{Wo} for general Alexandrov spaces with boundary), and hence has a unique point at maximal distance,
its ``soul point", $s_0$. It follows that, $s_0$ is in the intersection of
the remaining chamber faces (Theorem 1.3 in \cite{GKi} and Lemma 3.4 in \cite{Wo} for general Alexandrov spaces noting that intersections of faces are extremal subsets). Moreover,
by convexity of super level sets the distance function to
$s_0$ on $C$ has no critical points. Using this and the basic fact that convex combination of ``gradient like vector fields" is ``gradient like" one constructs a \emph{gradient like} vector field (the angle between it and any minimal geodesic to $s_0$ is larger than $\pi/2$) which is radial near
$s_0$ and gradient like also when restricted to the remaining faces
intrinsically. In particular, $C$ is the cone on $F_0$ which in turn
is isotopic to a small metric ball in $C$ of radius $\epsilon$
centered at $s_0$. This also identifies $F_0$ with the boundary of
this $\epsilon$ ball, which via the exponential map is identified
with the closure of a chamber in the unit sphere at $s_0$
corresponding to the reflections $\r_1, \ldots, \r_{\ell}$. The
proof of the remaining two bullets is now completed by induction on
the number of chamber faces.
\end{proof}

We point out that this proof is a special case of a general result
about orbit spaces of positively curved manifolds due to Wilking
\cite{Wi2} (Theorem 7), related more directly to $\Sigma/\W$ in our context
however. We have included it here not only to make the exposition
more self contained, but also because it illuminates the particular
structure we have here.

\smallskip

We now turn to the case where the section $\Sigma$ is a projective
space. In this case, we will analyze the situation in its universal
cover $\tilde \Sigma$. Specifically, for each mirror $\Lambda$ in
$\Sigma$ corresponding to a reflection $\r$, we consider its lift
$\tilde \Lambda$ to $\tilde \Sigma$. As noted in the proof the section proposition \ref{dif-sec} and remark \ref{invol}, $\r$ has two canonical lifts. One of them is a
reflection  $\tilde \r$ in $\tilde \Lambda$, the other has two
isolated fixed points and restricts to $\a$ on  $\tilde \Lambda$.
Here we define $\tilde \W$ to be the reflection group on $\tilde
\Sigma$ generated by all $\tilde \r$, where we use all $\r$ from
$\W$. Note, that by construction, any lifted mirror is preserved by
$\a$, and that $\a$ commutes with any element from $\tilde \W$.
Combining this with the previous lemma one derives, whether $M$ is
simply connected or not, the following:

\begin{lem}[Projective Chamber]\label{w-proj}
Assume $\Sigma$ is a $k$-dimensional projective space and $\tilde
\Sigma$ the universal cover with deck transformation $\a$. Then
\begin{itemize}
\item
Intersections of lifted mirrors are spheres invariant under $\a$.
\item
The associated reflection group $\tilde\W$ of $\W$ may or may not
contain $\a$, but in either case $\W = \langle \tilde \W, \a \rangle
/\langle \a \rangle$. \item Open chambers  $c$ in $\Sigma$ are
isometric to open chambers $\tilde c$ for $\tilde \W$,
\item
The closure $\tilde C$ of an open chamber for $\tilde \W$ is a
convex set in $\tilde \Sigma$ with boundary the union of chamber
faces. Moreover, $C$ is obtained from $\tilde C$ by identifying $\a$
orbits in the boundary.
\item
$\tilde C$ has at most $k+1$ chamber faces, and the intersection of
them all is $\Fix(\tilde \W)$.
\item
If $\tilde C$ has $k+1$ chamber faces it is a $k$-simplex and
$\Fix(\tilde \W) = \emptyset$.
\item
If  $\tilde C$ has $\ell+1 < k+1$ chamber faces it is a join of
$\Fix(\tilde \W)$ and an $\ell$ simplex.
\end{itemize}
\no Moreover, $\Sigma/\W = \tilde \Sigma/ \langle \tilde \W,\a
\rangle = (\tilde \Sigma/ \langle \tilde \W\rangle )/\langle \a
\rangle = \tilde C/ \langle \a \rangle$.
\end{lem}

 \medskip

 We assume from now on that $M$ is a positively curved simply connected polar $\G$ manifold, with $\G$ connected.

  The following is proved more generally for singular polar foliations in  \cite{AT}, Theorem 1.5 and \cite{Al}, Theorem 1.1:

\begin{thm}[Alexandrino and T\"{o}ben]\label{ata}
Any nontrivial polar action on a simply connected manifold has no
exceptional orbits and its reflection group $\W$ is the whole polar
group $\Pi$.
\end{thm}

 In the case of polar actions it was also recently proved in \cite{GZ} that in addition the chamber group is trivial. For the sake of the reader we provide a simple direct proof in the case of positive curvature. In fact, the following is pivotal for us:

  \begin{prop}[Chamber Group]\label{chamber group}
  The chamber group $\W_c$ of  a simply connected positively curved polar $\G$ manifold $M$ is trivial, and hence $M^* = \Sigma^* = C$. Moreover,

   \begin{itemize}
 \item
 If $\Sigma$ is a sphere, $C$ is a simplex and the fixed point set $\Sigma^{\W} = \emptyset$,  or $C$ is a join of $\Sigma^{\W}$ and a simplex.
 \item
 If $\Sigma$ is a projective space, $C$ is a simplex, $\a \in \tilde \W$ and  $\Sigma^{\W}$ is a subset of the vertices (possibly empty).
  \end{itemize}

  In either case,
  $\W$ acts simply transitive on the set of chambers.
 \end{prop}

 \begin{proof}
Consider an open chamber $c$ and $\W_c$ acting on it. Note that
whether or not $\Sigma$ is a sphere or a projective space, $c$ is
the union of compact closed locally convex subsets $C^{\epsilon}$
(distance $\epsilon$ or more to $C - c$). By convexity it is clear
that the soul point (the common soul point $s$ for all
$C^{\epsilon}$) is fixed by $\W_c$ (one can also use the description
of $c$ from the lemmas above). Since there are no exceptional orbits
when $M$ is simply connected (cf. \ref{ata}) this already is
impossible unless $\W_c$ is trivial. From section one we then know
that $M^* = \Sigma^*$ is the closure $C$ of a chamber $c$. If
$\Sigma$ is a sphere, Lemma \ref{w-sphere} completes the proof.

Now suppose $\Sigma$ is a projective space:

\no First note that $\a$ acts freely on the set of open chambers for $\tilde \W$. This follows from the simple fact that $\a$ interchanges the two connected components of the complement of any lifted mirror, and commutes with $\tilde \W$.

\no  We now claim that $\Sigma^{\tilde \W} = \emptyset$ and hence
$\tilde C$ (defined in \ref{w-proj}) is a simplex. Indeed, if $\Sigma^{\tilde \W}$ is nonempty
then clearly $\a \notin \tilde \W$. Moreover, the involution induced
by $\a$  on  $\tilde \Sigma/ \tilde \W =  \tilde C$ acts freely on
$\Sigma^{\tilde \W}$ and preserves the boundary of $ \tilde C$. In
particular, $ \tilde C/\langle \a \rangle$ will have interior metric
singular points contradicting that it is $C$ by Lemma \ref{w-proj}.

\no To complete the proof we now claim that $\a \in \tilde \W$, and
in particular $C = \Sigma/\W = \tilde \Sigma/\tilde \W =  \tilde C$.
If not, then $|\W| = |\tilde \W|$ and $\a$ induces a nontrivial
involution on $ \tilde C$ with  $C = \tilde C/ \langle \a \rangle$.
Such an involution will preserve the boundary of the simplex $\tilde
C$ taking faces to faces. As before this will produce an interior
metric singular point of $C$ unless the induced map by $\a$ is a
reflection of the simplex. This, however, cannot happen since the
fixed point set of this involution would correspond to a chamber
face of $C$ and hence a reflection in $\W$ whose lift to $\tilde
\Sigma$ had been omitted from $\tilde \W$.
\end{proof}

\begin{rem}\label{section fix points}
Note that it follows from this that if  $M^{\G}\ne \emptyset $ and
$\Sigma$ is a sphere then $M^{\G} = \Sigma^{\W}$, since
$\Sigma^{\W}$ is the most singular stratum in the orbit space
$\Sigma/\W = M/\G$. In the next section we will see that conversely,
if $\Sigma^{\W} \ne  \emptyset$ and $\Sigma$ is a sphere then
$M^{\G}  \ne \emptyset $ as well and hence $M^{\G} = \Sigma^{\W}$
(cf. \eqref{fix}).
\end{rem}

The Coxeter group $\W$, respectively $\tilde \W$ corresponding to the section being a sphere, respectively a projective space admits a canonical representation, i.e., acts isometrically on the unit $k$-sphere $\Sph^k$ with orbit space $C'$ having the same infinitesimal singularities (i.e., tangent cones) as $C$. As a consequence we have:

\begin{cor}
Let $M$ be a simply connected positively curved polar $\G$ manifold. Then $M^* = \Sigma^* = C$  admits a metric of constant curvature isometric to its linear model $C'$.
\end{cor}

 \begin{proof}
 From Proposition \ref{chamber group}, we know that $\Sigma^*$ is the chamber $C$ for a Coxeter group $\W$ (resp. $\tilde \W$) acting on the $k$-sphere $\Sigma$ (resp. $\tilde \Sigma$, when $\Sigma$ is a projective space). Moreover, as stated above the same Coxeter group acts linearly on $\Sph^k$, with chambers $C'$ having labels as $C$ and with the same  tangent cones, determined by
the corresponding isotropy groups and actions.

We only consider the case that $\Sigma$ is a sphere since the other case is analogous.
 Now fix a chamber $C$ with $\ell+1$ chamber faces, $F_0, \ldots, F_{\ell}$ in $\Sigma$ and the corresponding model chamber  $C' \subset \Sph^k$. As in the proof of Lemma \ref{w-sphere}, let $s_0$ be the point in $C$ at maximal distance to the face $F_0$. Now apply the isotropy group $\W_{s_0}$ of the Coxeter group at $s_o$ to $C$ to obtain a $\W_{s_0}$ invariant convex subset  $\W_{s_0}(C)$ of $\Sigma$ with $s_0$ in the interior, the point at maximal distance from the boundary $\partial (\W_{s_0}(C)) =  \W_{s_0}(F_0)$ of  $\W_{s_0}(C)$. As in the proof of Lemma \ref{w-sphere} it follows that there is a $\W_{s_0}$ invariant  smooth vector field on  an open neighborhood of $\W_{s_0}(C)$ in $\Sigma$, which is radial near $s_0$ and gradient like on $\partial (\W_{s_0}(C))$.

 The same construction based on $C'$ in $\Sph^k$ yields a $\W_{s_0}$ invariant diffeomorphism of a neighborhood
of $\W_{s_0}(C)$ in $\Sigma$ to a neighborhood of $\W_{s_0}(C')$ in
$\Sph^k$. After a suitable reparametrization of one of the vector
fields using transversality if needed, the restriction yields the
desired diffeomorphism from $C$ to $C'$.
 \end{proof}

We are now ready to establish the main result of this section.

\begin{thm}[Coxeter Section]\label{coxeter section}
Let $M$ be a simply connected positively curved polar manifold. Then
the action of the polar group $\W$ of a section $\Sigma$ is
equivariantly diffeomorphic to a linear action of  $\W$. In fact,
$\Sigma$ admits a $\W$ invariant metric of constant curvature.
\end{thm}

\begin{proof}
Choose a constant curvature metric on  $\Sigma^*$ as above. We now
claim that this metric comes from a $\W$ invariant metric on
$\Sigma$ with constant curvature. To see this, all we have to do is
to lift the metric locally near any point of the orbit space to any
point mapping to it by the orbit map. This however is clear. Since
the lifted metrics obtained this way agree on overlaps we are done.
\end{proof}

\begin{rem}
We point out that our conclusions about the section  in the theorem above carry over to the general context of a positively curved manifold with a nontrivial isometric reflection group action. The manifold together with the action of the reflection group is then equivariantly diffeomorphic to a sphere or a real projective space with a linear action by a finite Coxeter group ($\Z_2$ ineffective in the latter case).
\end{rem}

The following is now natural

\begin{definition}
We say that a simply connected positively curved polar $\G$ manifold
$M$ is \emph{reducible} if  the action by the Coxeter group $\W$ on $\Sigma$ is reducible.
\end{definition}

  In
particular it follows that $\W$, or $\tilde \W$ is an irreducible
Coxeter system group
 when ($M,\G$) is irreducible, but Example \ref{trivial examples}(3)  implies that the converse is false.
 Also an action with a nontrivial fixed point set is reducible.
 In the case of irreducible actions all the types $\A_n, \CC_n, \D_n, \E_6, \E_7, \E_8$ and $\F_4$ are of course possible when the section is a sphere, but we note that due to the Chamber Group Proposition above, not all of them are possible when the section is a projective space.

We remark that in the literature the notion \emph{hyperpolar} is
used for a polar manifold with flat sections. Following \cite{GZ}, we
say that a polar manifold is a \emph{polar space form} if its
sections have constant curvature. According to the sign of the curvature
of the sections one then says that the polar space form has
spherical, euclidean or hyperbolic type. Using this language, a
partition of unity argument as in \cite{GZ} Thm. 3.3, or the main result of \cite{Me} now yields the
following in our case

\begin{cor}[Polar Space Form]\label{polar space form}
A simply connected positively curved polar $\G$ manifold $M$ admits the structure of a
polar spherical space form with the same sections.
\end{cor}

It should be noted that $M$ with such a polar space form structure
typically has curvatures of both signs. In general, a highly
nontrivial result of \cite{Me} asserts that \emph{any metric} on a section
of any polar $\G$ manifold \emph{invariant under the polar group extends}
to a $\G$ invariant metric on the ambient manifold with the \emph{same
section}.

\bigskip

\section{The Chamber System and Primitivity}\label{chambersystem}   

Based on the Chamber Group Proposition \ref{chamber group}, recall
from Section \ref{sectionone} that there are two natural chamber systems
$\mathscr{C}(\Sigma, \W)$, respectively $\mathscr{C}(M, \G)$
associated with any polar action of a connected compact Lie group
$\G$ on a simply connected positively curved manifold $M$ with
section $\Sigma$ and polar group $\W$. Throughout the rest of the
paper $(M,\G)$ is such a polar pair.
\smallskip

Our primary purpose in this section is to analyze $\mathscr{C}(M,
\G)$ further and thereby derive essential properties about such
general actions. In particular, we will show that it is a
connected chamber system (the crucial starting point for our
subsequent investigation of irreducible actions), and use this to
show that $\G$ is generated by the face isotropy groups of any fixed
chamber $C \subset \Sigma$ (an essential ingredient in our
investigation of reducible actions).

When the chambers are spherical simplices, we observe that all proper residues of the chamber system can be described via slice representations of corresponding isotropy groups. This allows us to invoke a celebrated result of Tits \cite{Ti2}
implying that the so-called universal cover of our chamber system is
a building in most cases.

\bigskip
From the description $\mathscr{C}(M, \G) =  \cup_{g\in \G} g C$ of
the chamber system we first note that all chambers are isometric
when equipped with the induced length space metric from $M$. This
induces a natural length space metric on each path connected
component of $\mathscr{C}(M, \G)$.  A fundamental Theorem due to
Wilking \cite{Wi3} asserts in particular that the \emph{dual
foliation} associated to the orbits of an isometric group action on
a positively curved manifold has only one leaf. It is an immediate
consequence of this result that

\begin{itemize}
\item
$\mathscr{C}(M, \G)$ has only one component. \end{itemize}

There is an equivalent length metric on $\mathscr{C}(M, \G)$
obtained by using a polar space form metric on $M$ (cf. \ref{polar
space form}) in the construction above. We will refer to the
corresponding topology as the \emph{thin} topology on
$\mathscr{C}(M, \G)$. (Since $M$ is the union of its chambers, we
can also think of it as $M$ being equipped with this metric and
topology.)

From now on, we will always use the thin length metric on
$\mathscr{C}(M, \G)$ induced from a constant curvature one metric on
a section. In particular, note that then each chamber $C$ is either
a (spherical) $k$ - simplex $\Delta^k$, or else the spherical join
$\Sph^{k- \ell - 1} * \Delta^{\ell}$ of the $(k- \ell - 1)$-sphere
and a spherical $\ell$ - simplex. In either case, the chambers in a
fixed section $\Sigma$ tile the section, which is either $\RP^k$ or
$\Sph^k$.  Moreover, by construction, $\G$ preserves the
\emph{labeling} of all ``vertices, edges, \ldots, faces", i.e., of
all $0$-, $1$-, \dots, $(k-1)$-simplices, when $C = \Delta^k$ is a
simplex. In the special case where the chamber is not a simplex,
i.e., $C = \Sph^{k-\ell-1} * \Delta^{\ell}$, by a ``vertex", or
``0-simplex" of the chamber $C$ we mean a set of the type
$\Sph^{k-\ell-1} * \{v\}$, where $v$ is a vertex of the simplex
$\Delta^{\ell}$, and similarly for ``edges", \ldots, ``faces". We label
the set $\Sph^{k-\ell-1} \subset C$ as the $-1$-simplex of the
chamber $C$. In either case we note that the intersection of any two
chambers in $M$ is either empty or else a common ``subsimplex" in
this sense, allowing in particular the intersection to be a
``$-1$-simplex".

From the fact that $\mathscr{C}(M, \G)$ with the thin topology is
connected, we get the essential property:

\begin{thm}[Connectivity]\label{complex}
Assume $M$ is a simply connected positively curved polar $\G$ manifold.
Then the associated chamber system $\mathscr{C}(M;\G)$ is connected,
i.e., any two chambers are connected by a gallery.
\end{thm}

\begin{proof}
We will prove this by induction on dim$M^* = k$ using that $\mathscr
C(M;\G)$ is path connected.  For simplicity we first present the
proof in the typical case where the chamber $ C$  is a simplex
$\Delta^k$. A simple modification yields the general statement.

Let $C$ and $C'$ be two chambers of $\mathscr{C}(M;\G)$. Using
\cite{Wi3} join two interior points of $C$ and $C'$ by a piecewise
smooth \emph{horizontal curve}, i.e., at any point both one sided
derivatives of the curve are perpendicular to the $\G$ orbit at the
point. In our case, it is clear that we can choose a horizontal
curve $\gamma: [0,1] \to M$, and $0=t_0 < t_1 < t_2 \ldots <
t_{k+1}=1$ such that $\gamma_{| (t_{i}, t_{i+1})}$ is a geodesic, or
once broken geodesic in the interior of a chamber $C_i$ relative to
the thin metric on $\mathscr{C}(M;\G)$, where $C_0 = C$, $C_k = C'$
and all $C_i$ are different. Moreover, $\gamma$ can be chosen so
that each of the possibly nonsmooth points $\gamma(t_i)$, $i = 1,
\ldots k$ are all vertices. In addition, the one-sided
derivatives $\gamma'_+(t_i)$,  $-\gamma'_-(t_i)$ of $\gamma$ at the vertices $\gamma(t_i)$ are interior points of
two $(k-1)$ chamber simplices for the chamber complex
$\mathscr{C}(\Sph_{\gamma(t_i)}^{\perp};\G_{\gamma(t_i)})$ of the
slice representation of the isotropy group $\G_{\gamma(t_i)}$. By
induction these simplices can be joined by a gallery in
$\mathscr{C}(\Sph_{\gamma(t_i)}^{\perp};\G_{\gamma(t_i)})$. Filling
in the corresponding gallery in $\mathscr{C}(M;\G)$ at each
$\gamma(t_i)$ now yields a gallery from $C$ to $C'$.

To complete the proof we need to establish the induction anchor in
cohomogeneity two. By the same reasoning as above, this follows from
the claim that the chamber complex of a linear spherical
cohomogeneity one action is connected. Since any horizontal curve
provided by Wilkings theorem in this case is a piecewise horizontal
geodesic up to parametrization, such a curve already constitutes the
desired gallery.

The modification needed to cover the case where the chambers are
joins with a nonempty sphere can be explained as follows: As in the
simplex case one may choose a piecewise horizontal geodesic
$\gamma$, so that each of the possibly nonsmooth points points
$\gamma(t_i)$, $i = 1, \ldots k$ are most singular, i.e., in this
case $-1$-simplex points. The remaining part of the proof follows
the same path.
\end{proof}

The Coxeter Section Theorem \ref{coxeter section} and the
Connectivity Theorem above are the two crucial properties derived
using positive curvature. We note that there is no reason for the
chamber system of a simply connected polar space form of spherical
type to be connected. However:

\smallskip

The manifolds we actually classify in higher cohomogeneities in this
paper are the

\bigskip

\begin{center}
\emph{Chamber Connected Polar Spherical Space Form}
\end{center}

\no i.e.

$\bullet$ Simply connected polar space forms ($M,\G$) of spherical type

\no with

$\bullet$ Connected associated chamber system, $\mathscr{C}(M;\G)$

\medskip
In addition, this generality is important for the proof, because
$\G$ invariant polar submanifolds of a positively curved polar
manifold are typically not positively curved (cf. Section \ref{input}, proof
of Hopf Lemma).

\bigskip

\emph{The two assumptions above will be applied throughout the rest
of the paper.}

\smallskip

Using connectivity we derive the following simple but powerful tool:

\begin{thm}[Primitivity]\label{prim}
The group $\G$ is generated by the (identity components of the),  face isotropy groups of any fixed chamber.
\end{thm}

\begin{proof}
Fix a chamber $C_0$ and consider any other chamber $\g C_0$, $\g\in
\G$. Using the above, let $\Gamma = (C_0,....,C_k)$ be a gallery, of
type $i_1 i_2 \ldots i_k$, where $C_k = \g C_0$. By definition, note
that any $C_n$ is obtained from $C_{n-1}$ by applying an element
$g_{i_n}$ of the isotropy group for the common face $i_n$ of $C_n$
and $C_{n-1}$ to $C_{n-1}$, i.e., $C_n = \g_{i_n} C_{n-1}$. From this
it follows that  $C_k = \g C_0= \g_{i_k} \g_{i_{k-1}} \ldots \g_{i_1}
C_0$, and hence $\g = \g_{i_k} \g_{i_{k-1}} \ldots \g_{i_1}$ after
modifying $g_{i_1}$ with an element of the stabilizer of the chamber $C_0$
if necessary.

Now each $\g_{i_n}$ is a conjugate of an element of the isotropy group corresponding to the face $i_n$ by the previous element. So in other words $\g_{i_k} = [\g_{i_{k-1}} \ldots \g_{i_1}] h_{i_k}[\g_{i_{k-1}} \ldots \g_{i_1}]^{-1}$, and hence\\
$\g =  [\g_{i_{k-1}} \ldots \g_{i_1}] \h_{i_k}[\g_{i_{k-1}} \ldots
\g_{i_1}]^{-1} \g_{i_{k-1}} \ldots \g_{i_1} = [\g_{i_{k-1}} \ldots
\g_{i_1}] \h_{i_k}$, where  $\h_{i_k}$ is in the isotropy group with
face $i_k$ of $C_0$

Proceeding in this way we see that $\g = \h_{i_1}
\h_{i_2}.....\h_{i_k}$, where also $\h_{i_1} = \g_{i_1}$ as claimed. The claim about identity components of the face isotropy groups follows since these in fact act transitively on the normal spheres of their orbit strata (these spheres are connected).
\end{proof}

\begin{rem}\label{folding} The description of galleries used in the proof above is very useful.
In fact, a gallery starting at $C$ of type $i_1 i_2 \ldots i_k$ is
given by a word $ \h_{i_1}  \h_{i_2} \ldots \h_{i_k}$ in elements of
the isotropy groups $\G_{i_j}$ corresponding to the $i_j$-faces of
$C$. Note that each  $\G_{i_j}$ acts transitively on the normal
sphere to the corresponding orbit stratum, i.e., the $i_j$ residue
of $C$ is in one to one correspondence with this normal sphere. For
this reason we say that a gallery $\Gamma_f = (C_0,....,C_k)$ of
type $f=i_1 i_2 \ldots i_k$  is obtained from $C_0$ by
\emph{folding} it repeatedly along faces using the face isotropy
groups  $\G_{i_1}$,  $\G_{i_2}$, \dots,  $\G_{i_k}$.
\end{rem}

\begin{rem}\label{2vertex}
Note that this also immediately implies that $\G$ is \emph{generated by
any two vertex isotropy groups}.

\end{rem}

\begin{rem}\label{chamber homo}
We also observe that in complete generality, our chamber system  ${\mathscr{C}}(M;\G)$ associated to a polar $\G$ action on a simply connected  manifold $M$ is a \emph{homogeneous chamber system} of a type described in  Ronan's book \cite{Ro}. Specifically:

The
chamber system $\mathscr{C}(M;\G)$ is the left coset $\G/\H$ (the
principal orbit) with the following adjacency relation: two chambers
$g \H$ and $g'\H$ are $i$-adjacent if and only if $g \G_{
i}=g'\G_{i}$, where $\H$ is the principal isotropy group, and the
$\G_{i}$ for $i \in I$ are the face isotropy groups of a fixed chamber.

Note that  $\mathscr{C}(M;\G)$ is connected if and only if the polar $\G$ action is \emph{primitive}, i.e., by \emph{definition}: $\G$ is generated by the face isotropy groups.

\end{rem}

\smallskip

As promised we can use the above connectedness to prove the fixed
point claim from the previous section:

\begin{prop}\label{fix}
Suppose $M$ is a simply connected positively curved polar $\G$ manifold
with spherical section $\Sigma$ and polar group $\W$. Then $M^{\G} =
\Sigma^{\W}$, and in particular
$\rank(\W) = \dim \Sigma^* + 1 - \dim M^{\G}$.
\end{prop}

\begin{proof}
Since obviously $M^{\G} \subset \Sigma^{\W}$ and equality has been
proved in the previous section if $M^G$ is nonempty, it remains to prove that $M^{\G} \ne
\emptyset$ as long as $\Sigma^{\W}$ is nonempty (cf. \ref{section fix points}).

By assumption $M^* = \Sigma^* = C = \Sigma^{\W} * \Delta^{\ell} =
\Sph^{k-\ell-1} * \Delta^{\ell}$. Since all $\G$ orbits
corresponding to $\Sigma^{\W} = \Sph^{k-\ell-1}$, are of the same
type (corresponding to the most singular stratum of the orbit space) and are perpendicular to the section $\Sigma$ it suffices to
see that $\Sigma^{\W}$ is preserved by $\G$.

Pick any $\g \in \G$ and join the chamber $\g C$ to $C$ with a
gallery. Since any two consecutive chambers in a gallery have a
common ``face" and thereby the same ``$-1$-simplex", i.e., the same
fixed point set for the respective Weyl groups, it follows that also
$\g C$ has the same ``$-1$-simplex", which however is $\g
\Sigma^{\W}$.
\end{proof}

\begin{example}   \label{ex}
Here are examples showing that the conclusion above may fail in
cohomogeneity one as well as when the section is a projective space.

(1) Let $M = \CP^n = \SU(n+1)/\U(n)$. Then $\G = \U(n)$ acts by
cohomogeneity one with one fixed point. However, its polar group is
$\Z_2$ acting on a section $\Sph^1$ with two fixed points.

(2)  The obvious polar $\G = \U(1)\times \U(1)\times \U(n)$
representation on $\C^{n+2}= \C+ \C+ \C^n$ descends to a polar
action on $\CP^{n+1}$ with two fixed points (corresponding to the two
$\C$ summands). Its section is $\RP^2$ with $\RP^2/\W = \CP^{n+1}/\G$
a right angled spherical triangle. In particular, its Weyl group
must necessarily have three fixed points.
\end{example}

\smallskip

The case where the orbit space $M^* = \Sigma^* = C$ is not a simplex, i.e., by \ref{chamber group} it is a join of a
sphere with a simplex (in particular $M^{\G} \ne \emptyset$) will be dealt with in Section \ref{input}.

\medskip
We now point out some simple but crucial strong local properties of
the chamber system $\mathscr{C}(M, \G)$ of a positively curved
simply connected polar manifold in all remaining cases, i.e., when
the orbit space is a simplex.

\smallskip
Say $\M=(m_{ij})$, $i, j \in I$ is the Coxeter matrix for the
reflection group $\W$ of the section $\Sigma$ if it is a sphere, or
else of $\tilde \W$. In the latter case any word in the generators
$\r_i$ of $\W$ whose lift is the antipodal map in $\tilde \W$ is a
non-Coxeter relation in $\W$, and must necessarily involve all
generators of $\W$.  For any fixed  proper subset $J \subset I$ let $\M_J$ denote the submatrix of $\M$ with entries $m_{ij}$ , $i, j \in J$. Correspondingly, we  let $\W_J$ denote the subgroup of $\W$ generated by $\r_i$, $i \in J$. It is well known that  the subgroup $\W_J$ of $\W$ as well as of
$\tilde \W$, is a Coxeter group of type $\M_J$.
\smallskip

Recall, that a chamber system  $\mathscr{C}$ over $I$, by definition has \emph{type $\M$} if all $\{i,j\}$ residues, $i, j \in I$ are so-called generalized $m_{ij}$-gons (cf. \cite{Ro}).

 For any chamber $C$, consider $C_J := \cap_{i \in J}
C_i$, where $C_i$ is the $i$-face of  $C$. For
 an interior point $p \in C_J$, let $\Sph_{p,J}^{\perp}$ denote the unit sphere normal to the orbit stratum of $\G p$ at $p$, i.e.,  $\Sph_{p,J}^{\perp}$ is the sphere in the normal space to the orbit perpendicular to the fixed point subspace of $\G_p$. It is now apparent (see, e.g., \ref{folding}) that

\begin{lem}[Residue]\label{residue}
The $J$-residue of $\mathscr{C}$ and
$\mathscr{C}(\Sph_{p,J}^{\perp}, \G_p)$, for any $p\in C_J$ are
isomorphic as chamber systems of type $\M_J$.
\end{lem}

\medskip

Recall that a chamber system $\mathcal{B}$ over $I$ is called a
\emph{building} of type $\M=(m_{ij})$, $i, j \in I$, if each chamber
is $i$-adjacent to at least one other chamber, and there is a
$\W(\M)$ valued ``distance function"
$$\delta: \mathcal{B} \times \mathcal{B} \to \W$$
with the property $\delta(x,y) = \w \in \W$ if and only if the types
of minimal galleries between $x$ and $y$ coincide with the types of
minimal galleries in the Coxeter complex $\mathscr{C}(\Sigma, \W) =:
\mathcal{W}$ from $1$ to $\w$.

\medskip
The Coxeter complex $\mathcal{W}$ is itself a building with
$\delta(\u,\v) = \u^{-1}\v$. We call ``isometric" images of $\mathcal{W}$ in
$\mathcal{B}$ \emph{apartments} in $\mathcal{B}$. Another  example of central importance to us is the following

\begin{example}[Polar Representations]\label{polar rep}
The chamber system,  $\mathcal{B} = \mathscr{C}(\Sph, \K)$
associated to the restriction of a polar representation of a compact
Lie group $\K$ to the unit sphere $\Sph$ (without fixed points) is a
fundamental example of a (spherical) Tits building (see \cite{Ti1} and \cite{Da}).
\end{example}

\smallskip
\begin{rem}[Basic Building Properties] \label{buildingprop}
In a building $\mathcal{B}$, the following properties are basic  and
used repeatedly in the next sections.

\smallskip

$\bullet$ (Connectedness) Any two chambers $x, y$ are joined by a
minimal gallery $\Gamma_f$, which in turn is contained in an
apartment $A$.

$\bullet$ (Uniqueness) A minimal gallery from $x$ to $y$ is uniquely
determined by its type.

$\bullet$ (Convexity) If $x, y$ are chambers in an apartment $A$,
every minimal gallery from $x$ to $y$ is contained in $A$.

$\bullet$ (Homotopy) If $\Gamma$ is a gallery from $x$ to $y$ of
type $f$ (not necessarily minimal), and $f\simeq g$ (see below),
then there is a gallery of type $g$ from $x$ to $y$.

$\bullet$ A gallery of type $f$ is minimal if and only if $f =
i_1\cdots i_m$ is a so-called \emph{reduced word}, or equivalently
$\w = \r_f: = \r_{i_1}\cdots \r_{i_m}$ cannot be expressed as $\r_g$
for $g$ a shorter word.
\end{rem}
\smallskip

Since the slice representation of each isotropy group $\G_p$ is
polar, it follows from  \ref{polar rep} and the residue lemma \ref{residue}  that

\begin{prop}
For any proper $J \subset I$, any $J$ residue in the chamber system
$\mathscr{C}(M, \G)$ is a spherical building of type $\M_J$.
\end{prop}

 \no By invoking the following corollary of a profound result of Tits \cite{Ti2}, Corollary 3 in Section 5.3 (cf. also \cite{Ro}, Theorem 4.9), we get

 \begin{thm}[Tits]\label{tits covering}
 The universal Tits cover $\tilde {\mathscr{C}}$ of a (gallery-) connected chamber system $\mathscr{C}$ of (finite) type $\M$ over $I$ is a building
 if and only if all residues of rank three are covered by
 buildings.
 \end{thm}

 \no We conclude

\begin{thm}[Building Cover]\label{building}
Suppose $M$ is a positively curved simply connected polar $\G$ manifold with orbit space a simplex of dimension at least $3$.
 Then
the universal Tits cover $\tilde{\mathscr{C}}(M;\G)$ of the associated
chamber system $\mathscr{C}(M;\G)$ is a spherical building.
\end{thm}

\begin{rem}\label{free}
The fact that all residues of rank at least $3$ of the chamber system
$\mathscr{C}(M;\G)$ are buildings implies that universal Tits cover $\tilde{\mathscr{C}}(M;\G)$ can be viewed also as the usual topological universal cover  of $\mathscr{C}(M;\G)$ equipped with the thin topology. For this reason we frequently simply refer to $\tilde{\mathscr{C}}(M;\G)$ as the universal cover of $\mathscr{C}(M;\G)$. - In particular, the fundamental group $\pi$ of $\mathscr{C}(M;\G)$ \emph{acts freely} by deck transformations on
$\tilde{\mathscr{C}}(M;\G)$ equipped with the thin topology. It is a startling consequence of our main result Theorem \ref{mainresult}  in Section \ref{Compact spherical Buildings}  that $\pi$ in fact is either $\S^1$ or $\S^3$ with \emph{discrete topology} when $\M$ has no isolated nodes and $\tilde{\mathscr{C}}(M;\G)$ is a building. Note also that $\pi$ acting freely on the set $\tilde{\mathscr{C}}(M;\G)$ of course is independent on topology.
\end{rem}

Recall that the universal Tits cover is obtained via a notion of homotopies of galleries in
analogy with the usual construction of a topological universal cover.

Here two galleries $\Gamma_1\Gamma_0\Gamma_2$ and
$\Gamma_1\Gamma'_0\Gamma_2$ in a chamber system, $\mathscr{C}$  of
type $\M$ over $I$ are said to be \emph{elementary homotopic}  if
$\Gamma_0$ and $\Gamma _0'$ are galleries in a rank $2$ residue with
the same extremities. A {\it homotopy} from a gallery $\Gamma$ to
another one $\Gamma'$ (with fixed extremities) is a finite sequence of
elementary homotopies which transforms $\Gamma$ to $\Gamma'$. When
such a homotopy exists we write $\Gamma \simeq \Gamma'$.

 By construction, $\tilde {\mathscr{C}}$ as a set is a union of chambers, each chamber,
 $\tilde C \in \tilde {\mathscr{C}}$ being a {\it homotopy class}, $[\Gamma] = [C_0, \ldots, C_m]$ of galleries $\Gamma = (C_0, \ldots, C_m)$ from ${\mathscr{C}}$ starting at a fixed chamber
 $C_0 \in {\mathscr{C}}$ and ending at $C_m = C \in {\mathscr{C}}$, and where the covering map $p: \tilde {\mathscr{C}} \to {\mathscr{C}}$ takes $\tilde C$ to $C_m$.
 Also, the \emph{adjacency} relation among chambers is defined as follows: $\tilde C = [C_0, \ldots, C_m]$ is ``$i$-adjacent" to $\tilde C' = [C_0, \ldots, C_{m-1}, C_m']$ when $C_{m}$ and $C_m'$
 are  ``$i$-adjacent", and to $\tilde C'' = [C_0, \ldots, C_{m}, C'']$ for other $i$'s when $C_m$ and $C''$ are ``$i$-adjacent". All other \emph{incidence} relations follow from this,
  and the covering map $p$ preserves incidence relations.
   In
  this fashion the covering map $p$ preserves faces, and hence all other types.

\medskip

Note that in a Coxeter complex, $\mathcal{W}$ galleries starting at
$1$ are in one-to-one correspondence with their types. Here one also
uses the notion of \emph{strict homotopy}, denoted $f \simeq g$,
where the notion of an elementary homotopy  above is replaced by the stronger
notion of an  \emph{strict elementary homotopy}. Here a strict
elementary homotopy is an alteration of a word of the form
$f_1p(i,j)f_2$ to a word $f_1p(j,i)f_2$, where $p(i,j)$ is a word
of the form $\cdots ijij$ (with $ m_{ij}$ letters and ending in $j$);
e.g., if $m_{ij}=3$, $p(i,j)=jij$; and $p(j,i)=iji$. In particular,
$f$ and $g$ have the same length if they are strictly homotopic (but
not necessarily if they are just homotopic). Also $\r_f = \r_g$ if $f$ and $g$
are strictly homotopic, but the converse is false, since one may
have redundant letters; e.g, the words $f=f_1iif_2$ and $g=f_1f_2$
are not strictly homotopic but $\r_f = \r_g$. A word $f$ is called
{\it reduced} if it is not strictly homotopic to a word of the form
$f_1iif_2$.

\begin{rem}
Buildings are simplicial complexes, but our
chambers systems $\mathscr{C}(M;\G)$ are frequently not. This is
illustrated for example with the standard $\T^2$ action on $\CP^2$. Here all
chambers are spherical right angled 2-simplices, and they all have
the same three vertices, the fixed points of $\T^2$. As the Building Cover Theorem above
shows, we do not need to assume $\mathscr{C}(M;\G)$ to be simplicial when the rank of $\M$ is at least 4. However, in the rank 3 case where the Building Cover Theorem says nothing, we do indeed need  $\mathscr{C}(M;\G)$ to be simplicial in the irreducible cases, i.e., the cases of types $\A_3$ and $\CC_3$. This will be proved in Theorem \ref{simplicial}
and will allow us to use work of Tits on so-called \emph{geometries}, i.e., chamber systems of type $\M$ whose underlying geometric realization is simplicial.
\end{rem}

\begin{rem}
Equipped with the thin metric, our chamber system
$\mathscr{C}(M, \G)$ has the local structure of a CAT(1) space. This
is of course true for its universal cover $\tilde{\mathscr{C}}(M,
\G)$ as well. In fact, when its dimension is at least three
(corresponding to rank at least four), it follows by work of Charney
and Lytchak \cite{CL}, that in fact $\tilde{\mathscr{C}}(M, \G)$ is
a CAT(1) space and in fact a spherical building by their geometric
characterization of buildings.
\end{rem}

\smallskip

\section{Compact spherical Buildings} \label{Compact spherical Buildings}  

Throughout this section, we assume that the orbit space $M^* = C$ is
a simplex, and that the universal cover $\tilde {\mathscr{C}} :=
\tilde {\mathscr{C}}(M,\G)$ of our base chamber system $\mathscr{C}:
= \mathscr{C}(M,\G)$ is a spherical building of rank at least $3$. In particular, $\tilde {\mathscr{C}}$ is also a simplicial complex and we use $p: \tilde {\mathscr{C}} \to \mathscr{C}$ to denote the
covering map.

Our primary objective is to endow $\tilde {\mathscr{C}}$ with a
natural topology inherited from the topology of $M$, in such a way
that it becomes a compact spherical building in the sense of
 Burns and Spatzier \cite{BSp}, where the extension by Grundh\"ofer, Kramer,
Van Maldeghem, and Weiss in  \cite{GKMW} is crucial for us. Our second objective is to
analyze the fundamental group $\pi$ of ${\mathscr{C}}(M,\G)$ and its
action on the cover when $\tilde {\mathscr{C}}(M,\G)$ is a compact
spherical building. This in fact will imply Theorem A in the
introduction in all cases except where $\G$ has fixed points or
where the Coxeter diagram for $\M$  either has isolated nodes  or is
of type $\A_3$ or $\CC_3$.

 Section \ref{sectionC3gen} and  \cite{FGT} are devoted to the case where the
 Coxeter diagram of $\M$ is of type $\A_3$ or $\CC_3$.
In the special reducible cases where isolated nodes are present in the Coxeter diagram of $\M$ or $M^{\G} \ne \emptyset$,  rather different arguments will be employed in Sections \ref{input} and \ref{genRed}.

\medskip
We will write the set of vertices $\text{Vert}(\tilde
{\mathscr{C}})$ of a Tits building $\tilde{\mathscr{C}}$ as a
disjoint union $\text{Vert}(\tilde {\mathscr{C}})=\tilde{V}_1\cup\cdots\cup
\tilde{V}_{k+1}$ over the vertices of the same cotype where $k+1$ is the
rank of $\M$. The set of $r$-simplices of type $(i_1,\dots,
i_{r+1})$ for $r\le k$ will be denoted by  $\tilde
{\mathscr{C}}_{i_1,\dots, i_{r+1}}$.

 Recall, that a \emph{compact (spherical) building} according to \cite{BSp} is a Tits building $\tilde
 {\mathscr{C}}$ with a Hausdorff topology on the set
 $\text{Vert}(\tilde {\mathscr{C}})=\tilde{V}_1\cup\cdots\cup  \tilde{V}_{k+1}$ of all vertices
 such that the set $\tilde {\mathscr{C}}_{i_1,\dots, i_{r+1}}$ of all  simplices of type
 $(i_1,\dots, i_{r+1})$ is closed in the product  $\tilde{V}_{i_1}\times\cdots\times \tilde{V}_{i_{r+1}}$. With the
 induced topology on the $k$ simplices $\tilde {\mathscr{C}}_{1,\dots,k+1}$, $\tilde
 {\mathscr{C}}$ is called {\it compact, locally connected, infinite, metric } if $\tilde
 {\mathscr{C}}_{1,\dots,k+1}$ has the appropriate property.

 It is the main result of \cite{BSp} that an infinite,
 irreducible, locally connected, compact, metric, topologically Moufang  building of rank at least $2$ is classical. Namely, it is
 a Tits building associated to a noncompact real semisimple Lie group via the following description (cf. also proof of Theorem \ref{mainresult}):

 \begin{example}[Symmetric Spaces and Buildings]
 Let $\U$ be a connected noncompact real semisimple Lie group without center
 and $\K \subset \U$ a maximal compact subgroup (which is unique up to conjugation). The isometric action of $\U$ on the
 \emph{symmetric space} $N = \U/\K$ of nonpositive curvature induces a continuous action on the boundary at infinity, $ \Sph_{\infty} $, with the same orbits as those of the subaction by $\K$.
 Here the action by $\K$ is topologically equivalent to the isotropy representation of $\K$ on the unit sphere $\Sph_p$ at $p\in N$ with $\U_p = \K$.

 The isotropy representation of $\K = \U_p$ is polar with sections the tangent spaces of \emph{flats} through  $p \in N = \U/\K$. These flats at infinity are  \emph{apartments} of a
 (topological) building, $\mathscr{C}(\U)$ equivalent to $\mathscr{C}(\Sph_p, \U_p)$. One gets all apartments in the building in this fashion by letting $p$ go through all points of $N$.
 The group $\U$  is the  identity component of the (topologiocal) automorphism group  $\text{Aut}_\text{top}(\mathscr{C} (\U))$ of the building.

 An algebraic description of $\mathscr{C}(\U)$ can be given via the set of all parabolic subgroups of $\U$. We think of $\mathscr{C}(\U)$ as the set of parabolic subgroups of $\U$ with the following partial order:
If $C_1, C_2 \in \mathscr{C}(\U)$, we call $C_1$ a \emph{face} of $C_2$ and write $C_1 <  C_2$ if $C_2 \subset C_1$.
The chambers are the minimal parabolic subgroups.  Let $\W$ denote the Weyl group of the symmetric space $\U/\K$.
We fix a minimal parabolic subgroup $\B$.
The $\W$ valued metric $\delta$ in the definition of a building is then defined as follows: Given chambers
$C=\g\B$ and $C'=\g'\B$, there is by the Bruhat decomposition a unique $\w \in \W$ such that
$\B\g^{-1}\g'\B=\B\w\B$. We set $\delta(C,C') = \w.$

The correspondence between the geometric and algebraic description
is that the isotropy groups under $\U$ of the chambers and their
subsimplices at infinity are
 exactly the parabolic subgroups of $\U$.
 \end{example}


 \begin{center}
 The topology on $\tilde {\mathscr{C}}(M, \G)$
 \end{center}

 \smallskip

When considering ${\mathscr{C}}(M,\G)$ as a set of chambers, each
being a compact subset of the metric space $M$,
${\mathscr{C}}(M,\G)$ is a compact metric space with the classical
Hausdorff metric.
 Moreover, the same holds for the set of all galleries with any upper bound on the number of chambers. Since $\tilde{\mathscr{C}}(M,\G)$ is a building of type $\M$ any two chambers can be connected by a gallery of length at most $1/2|\W(\M)|$.
\medskip

Let us fix a chamber $\tilde C_0 \in \tilde{\mathscr{C}}$. For any
fixed large positive integer $k\ge \frac 12 |\W(\M)|$, $\epsilon
> 0$ and any chamber $\tilde C \in \tilde {\mathscr{C}}$, we let

\begin{center}
$B_{\epsilon, k}(\tilde C)$ be the union of those chambers $ \tilde
C' \in \tilde {\mathscr{C}}$
\end{center}

\no for which there are (stuttering) galleries $\Gamma$ and
$\Gamma'$ of length at most $k$
 starting at $\tilde C_0$ and ending at $\tilde C$, respectively $\tilde C'$ so that  the (stuttering) galleries $p(\Gamma)$ and $p(\Gamma')$ in ${\mathscr{C}}$ are
 within Hausdorff distance $\epsilon$ from one another in $M$. We will refer to the topology generated by these sets as the \emph{chamber topology} on the building
 $\tilde {\mathscr{C}}$.

 \smallskip
The geometric realization of the building $\tilde {\mathscr{C}}(M,
\G)$ is a simplicial complex $\tilde {\mathscr{S}}(M,\G)$. We will
show that this topology induces a topology on  $\tilde
{\mathscr{S}}(M,\G)$   making it into a compact spherical building
in the sense of Burns and Spatzier \cite{BSp}.

\smallskip
The following will be used repeatedly

\begin{lem}[Homotopy Control]\label{control}
Let  $\Delta$  be a building of rank at least $3$. Then for any $k$
there is a $C(k)$ with the following property: Any galleries
$\Gamma$ and $\Gamma'$ of lengths at most $k$ with the same
extremities are homotopic by a homotopy consisting of at most $C(k)$
chambers.
\end{lem}

\begin{proof} Since any building of rank at least three is simply
connected, $\Gamma$ and $\Gamma'$ are homotopic. The remaining part
of our claim is proved by induction on $k$, being trivially true for
$k=1$.

If $\Gamma $ and $\Gamma '$ are both minimal, the claim is a direct
consequence of the Convexity Property in \ref{buildingprop}.
Similarly, if $\Gamma = \Gamma_1\Gamma_0\Gamma_2$ and $\Gamma' =
\Gamma_1\Gamma'_0\Gamma_2$, where $\Gamma_0$ and $\Gamma_0'$ are
minimal (e.g., when there is a strict elementary homotopy from
$\Gamma$ to $\Gamma'$). In particular, by induction it suffices to
prove that a nonminimal $\Gamma$ is strictly homotopic to a
$\Gamma'$ via an a priory bounded number of strict elementary
homotopies and $\Gamma' $ is homotopic to a shorter gallery within a
uniformly bounded number of chambers.

Suppose $\Gamma$ is not minimal of type $f$. We claim that $\Gamma$
is strictly homotopic to a gallery $\Gamma'$ of type $f_1iif_2$
through at most $ \ell ^\ell$ strictly elementary homotopies, where
$\ell=|I|^k$. Indeed, the number of words of length at most $k$ is
bounded above by $|I|^k$. Therefore, the noncircuit operations from
a word of length at most $k$ to another one of length at most $k$ is
bounded above by $\ell ^\ell$.

Now,  a gallery $\Gamma' $ of type $f_1iif_2$ from $x$ to $y$ is
obviously homotopic to a shorter gallery of type either $f_1if_2$ or
type $f_1f_2$, according to the chambers being $\Gamma
_1C_1C_2C_3\Gamma_2$ (where $C_1\sim _i C_2$, and $C_2\sim _i C_3$)
or $\Gamma _1C_1C_2C_1\Gamma_2$ (where $C_1\sim _i C_2$). Moreover,
the homotopy can be realized in the longer gallery and so the number
of chambers is bounded by the length $k$.
\end{proof}

\begin{rem}
The proof of the above lemma gives an algorithm to construct a
\emph{controlled homotopy} between galleries with the same
extremities in a building.
\end{rem}

\begin{prop}\label{topology}
With the chamber topology, $\tilde {\mathscr{C}}$ is a compact,
separable and metrizable space.
 \end{prop}

\begin{proof}
By the  Uryson Characterization Theorem for metrizable spaces, all
we need to prove is that $\tilde {\mathscr{C}}$ is sequentially
compact, separable, and regular.

$\bullet$ (Sequential Compactness) Any sequence $\{\tilde C_n\}$ of
chambers in $\tilde {\mathscr{C}}$ has a convergent subsequence.

 For each $n$, let $\Gamma_n$ be a gallery of length at most $k$ joining
 $\tilde C_0$ and $\tilde C_n$.
 By compactness of $M$ the sequence $p(\Gamma_n)$ has a
 convergent subsequence in the Hausdorff metric topology with limit a  gallery $\bar{\Gamma} _\infty$ starting at $p(\tilde C_0)$.
  By the unique homotopy lifting property (see \cite{Ro}, Lemma 4.4), $\bar{\Gamma} _\infty$ can be uniquely lifted to a gallery, say $\Gamma _\infty$, starting at $\tilde C_0$.
  By the definition of the chamber
 topology we know that the corresponding subsequence of $\{\tilde C_n\}$ converges to the end chamber of $\Gamma _\infty$.

$\bullet$ (Separability) We may choose a countable dense subset
$Q_i$ of each face isotropy group $\G_i$, e.g.~the rational points.
 The set of galleries starting at $\tilde C_0$ of length at most $k$
obtained by the folding process described in \ref{folding} using
only elements from $Q_i$ is clearly dense in the set of all
galleries starting at $\tilde C_0$ of length at most $k$. By
definition, the last chamber of these lifted galleries in $\tilde
{\mathscr{C}}$
 starting at $\tilde C_0$ form a countable dense set in the chamber topology.

$\bullet$ (Regularity) We need to prove that, for a chamber $\tilde
C_1$ and  a closed subset $B \subset \tilde {\mathscr{C}}$ in the
complement of $\tilde C_1$, there are two disjoint open sets $U$ and
$V$ containing $\tilde C_1$ and $B$ respectively.

If this is not the case, we find for arbitrary large integers $n$, a
chamber $\tilde C'_n \in B_{\frac 1n , k}(\tilde C_1) \cap B_{\frac
1n, k}(B)$. By the above we know that the closed subset $B$ is sequentially
compact. Therefore, a subsequence of $\tilde C'_n$  converges to
some chamber $\tilde C_2\in B$. Therefore, there are two pairs of
sequences of galleries $\Gamma_{i,n}$, $\Gamma'_{i,n}$, $i = 1,2$,
starting at $\tilde C_0$ and ending at the chambers $\tilde C_1$,
$\tilde C_2$, respectively $\tilde C'_n$, with $d_H(p(\Gamma_{i,n}),
p(\Gamma'_{i,n})) < \frac 1n$.
 For each $n$, $\Gamma'_{1,n}$ and
$\Gamma'_{2,n}$ have the same extremities in the building $\tilde
{\mathscr{C}}$, and hence $\Gamma'_{1,n}\simeq \Gamma'_{2,n}$ and
$p(\Gamma'_{1,n})\simeq p(\Gamma'_{2,n})$, by a homotopy $H'_n$. By
Lemma \ref{control} we can assume that $H'_n$ is composed of an a
priory bounded number of chambers independent of $n$.  Taking
convergent subsequences, we can assume that $p(\Gamma_{i,n})$ as
well as $p(\Gamma'_{i,n})$ converge to the same galleries $\bar
\Gamma_{i, \infty}$, and that these are homotopic by a homotopy
$H'_{\infty}$. So on the one hand, by the unique homotopy lifting
property, $\bar \Gamma_{i, \infty}, i = 1,2$  lift to galleries with
the same end chamber in $\tilde {\mathscr{C}}$. On the other hand
they lift to galleries with end chamber $\tilde C_i,  i = 1,2$
respectively. A contradiction.
 \end{proof}

 \begin{lem}[Independence]\label{independence}
 The chamber topology is independent of the choices of $\tilde C_0$ and the parameter $k$.
 \end{lem}

 \begin{proof}
 Let us first prove the independence of $k$. If $k' > k$ clearly $B_{\epsilon, k}(\tilde C) \subset B_{\epsilon, k'}(\tilde C)$.  Consequently it suffices to show that a $k'$-convergent
 sequence of chambers $\{\tilde C_n\}$ is also $k$-convergent. By assumption there are galleries $\Gamma_{n}$ and $\Gamma^{n}$ in $\tilde {\mathscr{C}}$
  of length at most $k'$ starting at $\tilde C_0$ and ending at $\tilde C_n$ respectively $\tilde C$ such that the projected galleries $p(\Gamma_n)$ and $p(\Gamma^n)$ Hausdorff
converge to a gallery $\bar \Gamma_\infty$ (possibly stuttering) in
${\mathscr{C}}$. Again using Lemma \ref{control} we see that the
gallery $\Gamma_n$ is homotopic to a gallery $\Gamma'_n$ of length
at most $k$ by a homotopy $H_n$ with an a priory bounded number of
chambers. Note that $p(\Gamma _n')$ subsequentially converges to a
gallery $\bar \Gamma _\infty '= p(\Gamma _\infty ')$, where $\Gamma
_\infty '$ is the subsequence limit of $\Gamma'_n$. We may assume
the homotopies $p(H_n)$ also converge, and therefore we get a limit
homotopy between the two limit galleries $\bar \Gamma _\infty$ and
$\bar \Gamma _\infty'$. By the homotopy uniqueness lifting property
once again we get that $\Gamma _\infty$ and $\Gamma _\infty'$ have
the same ending chambers $\tilde C$. Therefore, $\{\tilde C_n\}$
also $k$-converges to $\tilde C$.

 To see the independence of the choice of $\tilde C_0$ join another chamber $\tilde C'_0$ to $\tilde C_0$ with a fixed gallery $\Gamma_0$, and the claim follows from independence
 of $k$ via concatenation with $\Gamma_0$ and its opposite.
 \end{proof}

We will now investigate the topology induced on the set of vertices
from the chamber topology. That topology in turn will induce a
topology on the geometric realization $|\tilde {\mathscr{S}}(M,\G)|$, of the simplicial complex  $\tilde {\mathscr{S}}(M,\G)$
associated to the building referred to as the \emph{thick} topology on  $\tilde {\mathscr{C}}$ from now on.
Assuming our chamber system $\tilde {\mathscr{C}}$ has rank $k+1$
corresponding to cohomogeneity $k$, for any $i\in I= \{0, \ldots,
k\}$,  consider the set $\tilde V_i$ of cotype $i$ vertices in
$\tilde {\mathscr{C}}$. Let $\pi _i: \tilde {\mathscr{C}}\to \tilde
V_i$ denote the obvious projection map. For each $i$, we equip
$\tilde V_i$ with the quotient topology.

\begin{lem}[Vertex Space]\label{vertices}
For any $i\in I$, the projection $\pi _i: \tilde {\mathscr{C}}\to
\tilde V_i$ is an open map, and $\tilde V_i$ is compact and
Hausdorff. Moreover, for any $x \in \tilde V_i$, the fiber $\pi
_i^{-1}(x) \subset  \tilde {\mathscr{C}}$ is the residue $\mathrm{Res}(x)$
in $\tilde {\mathscr{C}}$, which is compact, and the restriction of
the covering map $p: \tilde {\mathscr{C}}\to {\mathscr{C}}$ to this
residue is a homeomorphism to the residue $\mathrm{Res}(p(x))$ in
${\mathscr{C}}$.
\end{lem}

\begin{proof}
We begin with a proof of the last claim. By construction of $\tilde
{\mathscr{C}}$, $p$ provides an isomorphism between the residues as
subbuildings. We need to show that the chamber topology restricted
to the residue Res$(x)$ coincides with the Hausdorff topology of
Res$(p(x))$ in the manifold $M$.

Since $\tilde {\mathscr{C}}$ and ${\mathscr{C}}$ are both compact
and Hausdorff, and $p:  \tilde {\mathscr{C}}\to {\mathscr{C}}$
obviously is continuous, it remains to check that Res($x$) is closed
in $ \tilde {\mathscr{C}}$. Let $\{\tilde C_n\}, n = 1,2, \ldots$ be
a sequence of chambers in Res($x$) which converges in $ \tilde
{\mathscr{C}}$. Join a fixed chamber $\tilde C_0$ to $\tilde C_1$ by
a gallery $\Gamma$. Using that the residues are buildings, join each
$\tilde C_1$ to $\tilde C_n$ by a minimal gallery $\Gamma_n$ within
the residue. A subsequence of the projections to ${\mathscr{C}}$ of
the concatenated galleries clearly converges in the Hausdorff
topology, and the end chamber of the lift of the limiting gallery is
the limit of $\{\tilde C_n\}$, which as a consequence is in the
residue.

To show that $\tilde V_i$ is Hausdorff it suffices to show that $\pi
_i:  \tilde {\mathscr{C}} \to \tilde V_i$ is an open map and the
cotype $i$-adjacency is a closed relation, i.e. the subset
$$\{(\tilde C, \tilde C')\in  \tilde {\mathscr{C}}\times   \tilde
{\mathscr{C}}: \tilde C \text{ and } \tilde C' \text{ have common
cotype $i$ vertices}\}$$ is closed in the product topology. To show
the latter, let $(\tilde C_n, \tilde C_n')$ be a sequence converging
to $(\tilde C, \tilde C')$, where $\pi_i(\tilde C_n) = \pi_i(\tilde
C'_n)$. In particular $C_n$ and $C_n'$ share an $i$-vertex, and
$(C_n, C'_n)$ converges to $(C, C')$ in the Hausdorff topology of
$M$. Join $\tilde C_n$ to $\tilde C_n'$ by a minimal gallery
$\Gamma^i_n$ in the $i$-residue and pick a subsequence if necessary
so that the image galleries $p(\Gamma^i_n)$ in the residues in $M$
converge. Obviously the limit gallery joins $C$ to $C'$, and, in
particular, they share an $i$ vertex. It follows that $\tilde C$ and
$\tilde C'$ share the type $i$ vertex.

Let us prove that $\pi _i$ is open. For this we need to see that
$\pi _i^{-1}(\pi _i(U))$ is open in $\tilde {\mathscr{C}} $, where
$U$ is a finite intersection of $B_{\epsilon_i, k}(\tilde C_i)$'s.
Pick a chamber $\tilde D' \in \pi _i^{-1}(\pi _i(U))$, i.e. $\pi
_i(\tilde D') = \pi _i(\tilde C')$ for some $\tilde C' \in U$. We
need to find a neighborhood $U'$ of $\tilde D'$ so that for any
$\tilde D'' \in U'$ there is a $\tilde C'' \in U$ with $\pi _i(D'')
= \pi _i(C'')$.

Let $V$ be a finite intersection of $B_{\epsilon_j, k}(\tilde
C_j)$'s so that $\tilde C' \in V \subset U$. Since $\tilde C'$ and
$\tilde D'$ are cotype $i$ adjacent, they are in the same cotype $i$
residue (of some vertex), and they can be joined within this residue
by a gallery $\Gamma$ explicitly obtained by folding (see
\ref{folding})  $\tilde C'$ repeatedly along faces using face
isotropy groups (fixing the cotype $i$ vertex) in $M$ via $p :\tilde
{\mathscr{C}}\to \mathscr {C}$. To complete the proof the following
observation suffices: Consider the chamber $C' = p(\tilde C')$ in
$\mathscr {C}$. Any chamber Hausdorff close to $C'$ is $\g C'$ for
some $\g \in \G$ close to $1 \in \G$, and $\g\Gamma$ is thus close
to $\Gamma$.  Thus this process and its inverse takes a neighborhood
of $\tilde C'$ to a neighborhood $\tilde D'$ and conversely, and
 the claim follows.
 \end{proof}

Now we are ready to prove the first of our main results in this
section.

\begin{thm}[Compact Spherical Building]\label{top build}
The spherical building $\tilde {\mathscr{C}}(M,\G)$
with the topology on the set of vertices induced by the
 thick topology on the chambers is a compact spherical building if its
 rank is at least $3$.
 \end{thm}

 \begin{proof} We have seen in Lemma \ref{vertices} that the space
 $V_1\cup\cdots\cup  V_{k+1}$
 of vertices
 is Hausdorff.
 It is therefore left to show that
 the set $\tilde {\mathscr{C}}_{i_1,\dots, i_{r+1}}$ of all  simplices of type
 $(i_1,\dots, i_{r+1})$ is closed in the product  $\tilde V_{i_1}\times\cdots\times \tilde V_{i_{r+1}}$.
 It follows from Proposition \ref{topology} and Lemma
\ref{vertices} that  the product map $\prod \pi _{i_j}:
\tilde {\mathscr{C}}_{i_1,\dots, i_{r+1}}\to\tilde V_{i_1}\times\cdots\times \tilde V_{i_{r+1}}$
is continuous for any
multi-index $i_1,\dots, i_{r+1}$
 and
its image is a closed subset, which finishes the proof.
 \end{proof}

 It is clear from what we have proved so far that the compact spherical building in Theorem \ref{top build}
 is an infinite compact metrizable building. We can now apply the main results of \cite{BSp} or
 rather its generalization in \cite{GKMW} to compact spherical buildings that need not be
 locally connected.

\begin{thm}[Classical Building]\label{class build}
Assume the compact spherical building $\tilde {\mathscr{C}}(M,\G)$
has rank at least $3$ and its associated Coxeter diagram has no
isolated nodes. Then it is the building at infinity of a product $N$ of
irreducible symmetric spaces of noncompact type of rank at least
$2$. The topological automorphism
group $\mathrm{Aut}_\mathrm{top}(\tilde {\mathscr{C}})$ of the
building $\tilde {\mathscr{C}}(M,\G)$ is a real noncompact
semisimple Lie group with finitely many connected components and its
identity component is isomorphic to the identity component of the
isometry group of the symmetric space $N$.
 \end{thm}

  \begin{proof}
   It follows from Theorem 1.2 in \cite{GKMW} that $\tilde {\mathscr{C}}(M,\G)$ is the building
   at infinity of a product of irreducible symmetric spaces of rank at least $2$ and a locally
   finite Bruhat-Tits building of dimension at least two. The building at infinity of the
   Bruhat-Tits building is totally disconnected and can therefore be excluded since by
  Lemma \ref{vertices} the vertex residues are locally connected compact spherical buildings.
The claims about $\mathrm{Aut}_\mathrm{top}(\tilde {\mathscr{C}})$ follow from \cite{BSp}.
 \end{proof}

 We now prove a general theorem about the lifted $\tilde \G$ action and the free subaction of $\pi$ (cf. remark \ref{free}) on the simplicial complex $\tilde {\mathscr{C}}(M,\G)$ when equipped with the thick topology

 \begin{thm} [Compact Transformation Group]\label{comp group}
 Assume the spherical building $\tilde {\mathscr{C}}(M,\G)$ has rank at least
$3$ and is equipped with the thick topology. Then the deck
transformation group $\pi$ with the compact open topology is a
compact subgroup of the topological automorphism group ${\rm
Aut}_{\rm top}(\tilde {\mathscr{C}})$. Moreover, there is a compact
subgroup $ \tilde{\G}$ of ${\rm Aut}_{\rm top}(\tilde
{\mathscr{C}})$, such that $\pi\subset \tilde \G$ is a normal
subgroup with quotient $\tilde \G/\pi=\G$, whose action covers the
$\G$-action on ${\mathscr{C}}$.
\end{thm}

  \begin{proof}
It is a simple consequence of the Independence Lemma
\ref{independence} that every element of $\pi$ is a homeomorphism
with respect to the chamber and thick topologies.  In particular,
$\pi$ is a subgroup of the topological automorphism group
$\text{Aut}_\text{top}(\tilde {\mathscr{C}}$).
 We now prove that
$\pi$ is a closed subgroup of $\text{Aut}_\text{top}(\tilde
{\mathscr{C}}$). Let $f_n$ be a sequence in $ \pi$ that converges to
$f$ in $\text{Aut}_\text{top}(\tilde {\mathscr{C}})$ in the compact
open topology. In particular, $f_n(\tilde C)$ converges to $f(\tilde C)$ in
the chamber topology for every chamber $\tilde C\in \tilde
{\mathscr{C}}$. Notice that $p(f_n(\tilde C))=p(\tilde C)$.
Therefore, $p(f(\tilde C))=p(\tilde C)$, and it follows that $f$ is
in $\pi$. The compactness of $\pi$ follows since the orbit of $\pi$
is compact and the action of $\pi$ is free.

It is well-known that the $\G$-action on ${\mathscr{C}}$ lifts to a
covering group $\tilde \G$-action on $\tilde {\mathscr{C}}$, where
$\tilde \G$ fits in an extension (see \cite{Ro}, Exercise 8 in
Chapter 4)
$$1\to \pi \to  \tilde{\G}\to \G\to 1.$$

Once again, by the Independence Lemma \ref{independence}, we see
that $ \tilde{\G}$ is a subgroup of $\text{Aut}_\text{top}(\tilde
{\mathscr{C}})$ and as above one can check that it is closed, hence
also compact, since both $\pi$ and $\G$ are.
\end{proof}

\vskip 2mm

Combining these results we have the following main result about polar manifolds of positive curvature:

 \begin{thm} \label{mainresult}
  Any polar action of a compact connected Lie group $\G$ on a simply connected positively curved manifold $M$ whose associated chamber system is covered by a spherical building $\tilde{\mathscr{C}}$ of rank at least three and whose diagram $\M$ contains no isolated nodes
  is equivariantly diffeomorphic to a polar action on a
  compact rank one symmetric space, other than the Cayley plane.
\end{thm}

\begin{proof} It follows from Theorem \ref{class build} that the simplicial complex
 $\tilde {\mathscr{S}}$ as a set with the thick topology
  is a sphere that we will denote by $\Sph$. The compact subgroup $\pi$ of $\text{Aut}_\text{top}(\tilde
{\mathscr{C}})$ is a Lie group since $\text{Aut}_\text{top}(\tilde
{\mathscr{C}})$ is a Lie group by Theorem \ref{class build}.

We would like to show that $\pi$ is connected. We denote the identity component of $\pi$ by $\pi_0$.
   Clearly, $\pi _0$ acts freely on the sphere $\Sph$, and
there is a covering $\Sph /\pi _0\to \Sph/\pi =M$ whose fiber has the same
number of points as $\pi/\pi _0$. This is a contradiction since $M$ is simply connected. It
follows that $\pi=\pi _0$  and that $\pi$ and  $\tilde G$  are both compact and connected
subgroups, i.e., Lie subgroups of the identity component $\U$ of $\text{Aut}_\text{top}(\tilde
{\mathscr{C}})$.
As a consequence, $\tilde {\G}$ has a fixed point
in the symmetric space $\U/\K$, where $\K$ is a maximal compact
subgroup of the semisimple Lie group $\U$. Therefore, up to
conjugation we can assume that $\tilde {\G} \subset \K$ and it
follows
 that the action by $\tilde {\G}$  is topologically equivalent to a linear polar action orbit equivalent to the isotropy representation of $\K$
 on $\Sph$.
 Since the  action of $\pi$ on $\Sph$ is both linear and free, $\pi$ is either $\{1\}$, $\S^1$ or $\S^3$ (cf. e.g. \cite{Br}) and by representation theory the action is the Hopf action.
 It follows that $M$ is
$\G$-equivariantly homeomorphic to the rank one symmetric space
$\Sph/\pi$ with the linear polar action by $\G=\tilde \G/\pi$.

To complete the proof, we note that the induced linear polar action
on $\Sph/\pi$ by $\G = \tilde{\G}/\pi$ has the same data,  i.e.,
section, polar group, isotropy groups and their slice
representations as the polar $\G$ action on $M$. From the
reconstruction theorem of [GZ] it follows that $(M,\G)$ is smoothly
equivalent to $(\Sph/\pi, \G)$.
\end{proof}

In view of the Building Cover Theorem \ref{building}, this takes care of all
cases where $\G$ has no fixed points and $\M$ has no isolated nodes
and rank at least 4.

\smallskip

We conclude this section by another application of Theorem \ref{mainresult}. As mentioned in the
introduction there are  polar $\G$ actions by $\SU(3)\SU(3)$ and $\SO(3)\G_2$ on
$\Bbb{OP}^2$ (see \cite{PTh, GK}) whose associated chamber systems
$\mathscr{C}(\Bbb{OP}^2, \G)$ are of type $\CC_3$. In
particular we conclude from Theorem \ref{mainresult}  that

\begin{cor}[Not a Building]\label{Non Building}
 The universal covers of the chamber systems $\mathscr{C}(\Bbb{OP}^2, \G)$ associated to the polar actions on
 $\Bbb{OP}^2$ by $\G = \SU(3) \cdot \SU(3), \SO(3) \cdot \G_2$ are simply connected chamber systems of type $\CC_3$ that are not buildings.
 \end{cor}

Examples of simply connected chamber systems of type
$\CC_3$ that are not buildings were discovered by Neumaier and later
but independently by Aschbacher. The examples of the chamber systems in
 \ref{Non Building}  are new and follow also from \cite{Ly} and \cite{KL} as noted by them. These intriguing examples motivate the
following interesting problems.

\begin{problem} [Cayley plane chamber system]\label{cayley} Let $\tilde {\mathscr{C}}$ denote the universal cover of the chamber system $\mathscr{C}: =\mathscr{C}(\Bbb{OP}^2;
\G )$, where $\G$ is one of  $\SU(3) \cdot \SU(3), \SO(3) \cdot \G_2$.
\begin{enumerate}
\item
 Is $\mathscr{C}$ itself simply connected?
\item
If $\mathscr{C}$ is not simply connected, does the section $\Bbb
{RP}^2$ lift to $\Bbb S^2$ in $\tilde {\mathscr{C}} $? What is its fundamental group, and what is $\tilde {\mathscr{C}}$?
\end{enumerate}
\end{problem}

\smallskip

\section{Irreducible Chamber systems and Tits geometries of rank $3$}  \label{sectionC3gen} 

The purpose of this section is to develop and describe an
alternative to the Building Cover Theorem \ref{building} for irreducible polar
actions of cohomogeneity two, i.e., for rank 3 chamber systems
$\mathscr{C} = \mathscr{C}(M, \G)$, where $\M$ has no isolated
nodes, or equivalently $\M$ has type $\A_3$ or $\CC_3$. In this
case, any  closed chamber $C$ of $\mathscr{C}$, or equivalently the
$\G$ orbit space of $M$ is the spherical triangle with angles
\{$\frac \pi 3$, $\frac \pi 2$, $\frac \pi 3$\}, or  \{$\frac \pi
4$, $\frac \pi 2$, $\frac \pi 3$\} respectively.

Our method is based on a \emph{construction of chamber system
covers} (corresponding to the principal bundle construction for
polar manifolds in \cite{GZ}), and on an \emph{axiomatic
characterization} due to Tits of buildings of irreducible type $\M$,
when the geometric realization $|\mathscr{C}|$ ($\mathscr{C}$ with
the thin topology) of the associated chamber system  $\mathscr{C}$,
is a \emph{simplicial complex}.  This characterization is given in
terms of the \emph{incidence geometry} associated with
$\mathscr{C}$. Here, by definition

\begin{itemize}
\item
Vertices $x,y \in |\mathscr{C}|$ are \emph{incident}, denoted $x*y$,
if and only if $x$ and $y$ are contained in a closed chamber of
$|\mathscr{C}|$.
\end{itemize}

Clearly, the incidence relation (not an equivalence relation) is
preserved by the action of $\G$ in our case.

To describe the needed characterization, and to prove that our
chamber systems $\mathscr{C}(M, \G)$ of types $\A_3$ and $\CC_3$ are
simplicial, we will use the following standard terminology:

\begin{itemize}
\item
The {\it shadow} of a vertex $x$ on the set of vertices of type $i
\in I$, denoted $\text{Sh}_i(x)$, is the union of all vertices of
type $i$ incident to $x$.
\end{itemize}

When $\M = \CC_3$,  we will use $q$, $r$, and $t$ respectively,  to
denote the vertices of a chamber $C$ at angles $\frac \pi 4$-,
$\frac \pi 2$-, and $\frac \pi 3$ respectively, corresponding to the
three nodes from left to right of the $\CC_3$-diagram

\begin{center}
\begin{picture}(30,10)
    \put(0,0){\circle{3}}
    \put(19,0){\circle{3}}
    \put(41,0){\circle{3}}


    \put(2,0){\line(2,0){16}}

    \multiput(20,-1)(0,2){2}{\line(2,0){20}}

\end{picture}
\end{center}

\medskip

\no The faces in $C$ opposite $q$, $r$ and $t$ respectively, will be
denoted by $\ell _q$, $\ell _r$ and $\ell _t$ respectively.

Following Tits \cite{Ti2}, we call the vertices of type $q$, $r$ and
$ t$,  {\it points, lines}, and \emph{planes}  respectively. We denote
by $Q, R$ and $T$ the set of points, lines, and planes in
$\mathscr{C}(M;\G)$. Notice that $\G$ acts transitively on $Q$, $R$
and $T$.

Using this terminology we prove the following key

\begin{thm}[Simplicial] \label{simplicial}
The geometric realization  $|\mathscr{C}(M,\G)|$ of a chamber system
$\mathscr{C}(M,\G)$ of type $\A_3$ or $\CC_3$ associated with a
simply connected polar $\G$-manifold $M$ is simplicial.
\end{thm}

\begin{proof} Since the case of $\A_3$ follows directly from a part of the proof of the $\CC_3$ case (cf. Case (i) below), we only discuss the latter.

We claim that all we need to show is that \emph{vertices of
different types are joined by at most one minimal geodesic}.
In particular, an edge is determined by its vertices.
In
fact, given this we only need to prove that any chamber $C$ of
$|\mathscr{C}(M,\G)|$ is uniquely determined by its vertices.
So
suppose $C$ and $C'$ are chambers with the same vertices. From the
claim they have the same edges as well. Now by transitivity there is
a $\g \in \G$ with $\g C = C'$. Since $\g$ fixes all vertices and
edges of $C$ it is in the principal isotropy group of $C$ and hence
$\g C = C$.

\smallskip

Case (i). One of the vertices is a plane.

For a plane $t\in T$, note that the shadow $\text{Sh}_Q(t)$
(resp. $\text{Sh}_R(t)$) of $t$ in $Q$ (resp. $R$) is the
homogeneous space  $\G_t q = \G_t/\G_t\cap \G_{q}$, where $q\in
\text{Sh}_Q(t)$ (resp. $q\in \text{Sh}_R(t)$) . Moreover, the set of
all edges containing $t$ and $q$ is the homogeneous space $\G_t\cap
\G_q \ell_r = \G_t\cap \G_q/\G_{\ell_r}$, where $\ell_r$ is a
minimal reference geodesic connecting $t$ and $q$.  It suffices to
prove that $\G_{\ell_r}=\G_t\cap \G_{q}$:

Consider the fibration
$$\frac{\G_t\cap
\G_{q}}{\G_{\ell_r}}\to \frac{\G_t}{\G_{\ell_r}}\to
\frac{\G_t}{\G_t\cap \G_{q}}$$

\no Note that the
base  cannot be a point, since otherwise,
$\G_t\subset \G_{q}$, and so by the primitivity $\G=\langle \G_t,
\G_{q}\rangle =\G_{q}$, and hence $\G$ would have fixed points. On
the other hand, $\G_t/\G_{\ell_r}$ is the set of points (resp.
planes) in a type $\A_2$ geometry, associated with the slice
representation at $t$, i.e., $\G_t/\G_{\ell_r}=\Bbb P^2(k)$, where
$k=\Bbb R, \C, \Bbb H$ or $\Bbb O$. In particular, $\G_{\ell_r}$ is a maximal subgroup of $\G_t$, and thus $\G_{\ell_r}=\G_t\cap \G_{q}$.

\smallskip
Case (ii). One of the vertices is a line.

Consider a chamber $C$ with sides $\ell_t$, $\ell_r$, and $\ell_q$, and suppose $\ell'_t$ is another minimal geodesic joining the vertices $r$ and $q$ of $C$. Since each singular isotropy group of the reducible slice representation of $\G_r$ acts transitively on the other singular orbit, there is a $\g \in \G_{\ell_q} \subset \G_r$ with $\g \ell_t = \ell'_t$. By (i) $\g C$ is a chamber with sides $\ell'_t$, $\ell_r$, and $\ell_q$. But since $\g$ fixes $\ell_r$ and $\ell_q$ it is in the principal isotropy group of the slice representation of $\G_t$ and hence $\g C = C$. Thus $\ell'_t = \ell_t$.
\end{proof}

Since by work of Tits \cite{Ti2} (cf. Proposition 6), any
$\A_n$-geometry is a building, we conclude as in Theorem \ref{mainresult}  (with
$\pi$ trivial) that

\begin{cor}
A simply connected positively curved  polar $\G$-manifold of type
$\A_3$ is equivariantly diffeomorphic to a polar $\G$ representation
on a sphere.
\end{cor}

 \smallskip

In the much more complicated and rich case, where the chamber system $\mathscr{C}(M,\G)$ is of type $\CC_3$ our classification carried out in \cite{FGT} hinges on an axiomatic characterization for a  connected  Tits geometry of type $\CC_3$ to be a building \cite{Ti2}.

For all but two such chamber systems, this Tits axiom is verified for a suitable cover of $\mathscr{C}(M,\G)$, and the two exceptional cases are identified with the chamber systems for the two exceptional polar actions of type $\CC_3$ on the Cayley plane. - Since by Theorem \ref{simplicial} $\mathscr{C}(M,\G)$ is simplicial an alternative proof is offered in \cite{KL}.

\section{Reduction Input and Fixed Point Case} \label{input}   

In the last two sections we will deal with reducible polar
 actions in positive
curvature.

The key result in this section is a characterization of Hopf
fibrations in our context, that also will play an essential role in
the next section. As a corollary we obtain a classification when
fixed points are present. We need the following

\begin{lem}[Extension]\label{ext}
Let $(\Sph^n, \G)$ be a fixed point free effective polar representation with associated chamber system (building) $\mathscr{C}(\Sph^n;\G)$. If $\tilde {\G }\supset \G$ is a compact connected subgroup of  ${\rm Aut}_{\rm top}(\mathscr{C}(\Sph^n,\G))$, then $\tilde {\G }$ is a Lie group acting linearly on $\Sph^n$.
\end{lem}

\begin{proof}
First note that the induced action by $\tilde{\G}$ on $\Sph^n$ is continuous and orbit equivalent to the $\G$ action.

We begin by considering irreducible $\G$ representations. In the special case where $\G$ acts transitively on $\Sph^n$, a chamber of $\mathscr{C}(M;\G)$ is just a point in $\Sph^n$ and ${\rm Aut}_{\rm top}(\mathscr{C}(\Sph^n,\G))$ is the homeomorphism group of $\Sph^n$ with the compact open topology. Since all $\tilde{\G}$ orbits (there is only one) are locally connected and $\Sph^n$ a manifold, Theorem 1 (page 244) of \cite{MZ} states that $\tilde{\G}$ is a Lie group. In fact by \cite{Po}, $\tilde {\G }$ is a subgroup of $\SO(n+1)$.

In the case where $\G$ acts by cohomogeneity one or higher on $\Sph^n$, it follows from \cite{Da} that the action is orbit equivalent to the isotropy representation of a symmetric space $\U/\K$ of noncompact type. Moreover, $\mathscr{C}(\Sph^n;\G)$ is the building at infinity of $\U/\K$ whose (topological) automorphism group is $\U$, a Lie group. Thus, $\tilde{\G}$ is a compact subgroup acting isometrically on $\U/\K$, and hence with a fixed point, where the action is linear and orbit equivalent to the $\G$ action.

In general, the $\G$ action splits into a sum of irreducible subactions. From the above we conclude that the restriction of $\tilde{\G}$ to each subspace sphere is linear. Moreover, since $\tilde{\G} \subset {\rm Aut}_{\rm top}(\mathscr{C}(\Sph^n,\G))$, it takes chambers to chambers, and hence maps any minimal geodesic between $\G$ invariant subspace spheres to a minimal geodesic between the same invariant spheres. Thus, $\tilde{\G}$ acts linearly in fact isometrically on $\Sph^n$.
\end{proof}

\bigskip

We are now ready to prove

\begin{lem}[Hopf fibration]\label{hopf}

Let $(\Sph^n,\G)$ be a fixed point free linear polar action, and $(B, \G)$ a simply connected closed polar manifold.  Suppose $p: \Sph ^n \to B$ is a smooth, $\G$-equivariant, chamber preserving map with the following property: For each $v\in \Sph^n$, the differential $p_*$ on the normal slice at $v$ is a $\G_v$-equivariant isomorphism onto the normal slice at $p(v)$, orbit equivalent to the slice representation of $\G_{p(v)} \supset \G_{v}$. Then  $p$ is either a diffeomorphism, or a Hopf fibration up to equivariant diffeomorphism of $B$ (in particular the fibers are great spheres). Moreover, if $\dim B < n$ and $B$ is a sphere the cohomogeneity is at most $1$.
\end{lem}

\begin{proof}
Note that by assumption the chambers $C$ in $\Sph ^n$ and $B$ are
spherical $k$-simplices, where $k \ge 0$ is the cohomogeneity of the
actions, and $p$ is surjective. Moreover, $p$ is a submersion, since the differential $p_*$ on the tangent space to an orbit is surjective, and by
the assumption about slice representations $p_*$ is an isomorphism on the normal space to the orbit. Furthermore, $p$ when restricted to a section $\Sigma$ in $\Sph ^n$ is a cover of a section in $B$. This in particular proves our claim when $\dim B = n$.

When $\dim B < n$ we know from \cite{Br} that the fiber of the submersion $p$ is homeomorphic to $\Sph^i$, $i = 1, 3$ or $7$, where $i=7$ can only happen when $n=15$. Moreover, from the Gysin exact sequence applied to the fibration $p: \Sph^i \to \Sph^n \to B$,  it follows that, $B$ has the integral cohomology ring of a projective space $\Bbb{FP}^m$, $m \ge 1$ with $\Bbb F=\Bbb C$ or $\Bbb H$ when $i\in\{1,3\}$ and  $B$ is  a homotopy $8$-sphere if $i=7$.

Our proof for the case $\dim B < n$ is anchored at irreducible polar $\G$ representations and polar representations of cohomogeneity at most one (on $\Sph^n$) in conjunction with the above Lemma \ref{ext} and the Compact Transformation Group Theorem \ref{comp group}.
\smallskip

$\bullet$ Cohomogeneity $k = 0$.

\smallskip

From the list of $\G$ acting transitively and isometrically on $\S^n = \G/\H$ it follows directly (and is well known), that $\Sph^n = \G/\H \to \G/\K$ with fiber $\K/\H = \Sph^i$, $i \in \{1,3,7\}$ and $\G/\K$ simply connected, is a Hopf fibration (cf., e.g., Table C of \cite{GWZ}).

\smallskip

$\bullet$ Cohomogeneity $k = 1$.

Recall that a cohomogeneity one action and manifold $(M,\G)$ is completely determined by its \emph{data}, i.e., $\G$ and its isotropy groups along a chamber $C$ in $M$. Indeed, if $\K_{\pm}$ are the isotropy groups at the end points, $u_{\pm}$ of $C$ and $\H$ the principal isotropy group along the interior of $C$, then $\K_{\pm}/\H = \Sph^{\ell_{\pm}}$ are canonically identified with the normal spheres to the orbits at the end points of $C$, and via the slice theorem and canonical gluing

\begin{center}
$M = \G\times_{\K_-} \Disc^{\ell_-} \cup_{\G/\H} \G\times_{\K_+} \Disc^{\ell_+}.$
\end{center}

In our case, let $M = \Sph^n$ with data denoted as above. If $\K'_{\pm}$ and $\H'$ are the (local) data for $B$ along the chamber $C' = p(C)$ in $B$ with endpoints $u'_{\pm}$, then

\begin{center}
$B = \G\times_{\K'_-} \Disc^{\ell_-} \cup_{\G/\H'} \G\times_{\K'_+} \Disc^{\ell_+}$
\end{center}

\no where we have used our assumption $\K'_{\pm}/\H' = \Sph^{\ell_{\pm}}$, and moreover,
$\K'_{\pm}/\K_{\pm} = \H'/\H = \Sph^{i}$ are the fibers of $p$ along $C'$. It is important, that $p$ is determined by these data as well. We will refer to the dimensions, $(\ell_-, \ell_+)$ of the normal spheres of the singular orbits as the \emph{multiplicities} of the action.

\smallskip

$\bullet$ Fiber dimensions $i=1, 3$:

We point out that $B$ is already known up to equivariant diffeomorphism: Indeed, note that in the classification of cohomogeneity one actions on manifolds with the \emph{rational cohomology ring} of $\Bbb{FP}^m$ due to \cite{uchida} and \cite{iwata}, the quadric $\SO(2m+ 1)/\SO(2) \times \SO(2m-1)$ and $\G_2/\SO(4)$ are excluded in our case since they do not have the correct integral cohomology ring.
Thus, from \cite{uchida} and \cite{iwata} we conclude that the $\G$ action on $B$ modulo $\K_0$, the identity component of its kernel, is equivariantly diffeomorphic to a linear action on $\Bbb{FP}^m$, i.e.,  there is an  equivariant diffeomorphism $f: (B; \bar \G) \to (\Bbb{FP}^m; \bar \G)$, where the latter action is $\Bbb F$-linear, $\bar \G$ is the connected normal subgroup of $\G$ such that $\G=\K_0\cdot \bar \G$,  a product up to finite central quotient.

\smallskip

Consider first the case where the $\G$-representation is irreducible. A classification of these (including their data (corrected in \cite{FGT})) is contained in Table E from \cite{GWZ}. A corresponding classification of those induced on $\Bbb{FP}^m$ is contained in Table F in \cite{GWZ}.

When $i=3$, i.e., $B \cong \HP^m$, such actions have multiplicity pairs $(1,1)$ and  $(2,2l+1)$, where $m=1$ and $m=l + 1$, $l \ge1$, respectively. In the first case there is only one such action, while in the second there are two orbit equivalent actions. From the list of possible $\G$ actions on $\Sph^n$ with these multiplicity pairs, it necessary follows that $\K_0 = \S^3$ acting freely in a linear fashion along the fibers of $f\circ p$ as a subaction of $\G$,  and we are done. The same argument works when $i=1$ and the multiplicity pair is $(1,l)$,  including an ``exceptional case" for each of $(1,5)$ and $(1,6)$. In the remaining cases corresponding to the multiplicity pairs $(2, 2l+1)$, $(4,5)$ and $(9,6)$, either $\K_0 = \S^1$ and we are done, or $\K_0$ is trivial. In the latter case it follows that $\G$ acts almost effectively on $B \cong \CP^m$. This on the other hand determines all data, and hence $f \circ p: \Sph^{2m+1} \to \CP^m$ is the Hopf map as claimed.

 Now assume the $\G$ representation is reducible with singular orbits $\Sph_{\pm}$ and $\Sph^n = \Sph_- * \Sph_+$. From the homogeneous case it follows that the fibers of $p$ restricted to $\Sph_{\pm}$ are the fibers of a Hopf fibration. We will show that any fiber of $p$ is a fiber of the uniquely determined Hopf fibration of $\Sph_- * \Sph_+$ restricting to the given ones on $\Sph_{\pm}$. For any regular point $u' \in B$ let $C'$ be the unique chamber containing $u'$, and let $\K'_{\pm}$ be the isotropy groups at the end points $u'_{\pm}$ of $C'$. Also let $C$ be a chamber in $\Sph^n$ with $p(C) = C'$ having isotropy groups $\K_{\pm}$ at its end points $u_{\pm} \in \Sph_{\pm}$, and set $\K': = \K'_- \cap \K'_+$. From our assumption about slice representations and the natural identifications of the normal spheres to the singular orbits at say $u'_+$ and $u_+$ with $\Sph_{-}$ (and vice verse with $+ -$ swapped) it follows that the exponential map from $\Sph^{\perp}_{u'_{+}} \to p(\Sph_-) = B_-$ is a submersion with fibers identified with the fibers of $p_{| \Sph_{-}}$. Since the same is true with the roles of $+ -$ switched, it is not hard to see that $\K'(C') \cong \Sph^{i+1}$ is a  cohomogeneity $1$, $\K'$-submanifold  of $B$ with a suspension action, and $p^{-1}(\K'(C')) = \K'(C) \cong \Sph^i*\Sph^i$ is a cohomogeneity $1$, $\K'$-submanifold  of $\Sph^n$, where the two $\Sph^i$ in the join decomposition are the two Hopf fibers $p^{-1}(u'_{\pm})$.   Therefore, it suffices to establish our claim when $p:  \Sph^{2i+1}= \Sph^i*\Sph^i \to B \cong \Sph^{i+1}$ and the action on $B$ is a suspension action.

 When $i=1$, modulo kernel, the $\K'$ action reduces to the reducible $\T^2$ action on $\Sph^3$, and the suspension action on the base $\Sph^2$. Therefore, the kernel of $\T^2$ on the base is $\S^1\subset \T^2$, acting freely and linearly along the fibers of $p$, and we are done.

  When $i=3$, and if the kernel of the action on the base  $\Sph ^{4}$    act transitively on the fibers of $p$, the desired result follows as in the previous case. If not,  one checks easily that,  modulo kernel, the $\K'$ action contains the sum action of $\S^3\times \S^3$ on $\Sph^7$ as a subaction commuting with a Hopf action given by an $\Bbb H$-structure. This determines all data and $p$ is a Hopf map, with fibers the orbits the diagonal subgroup $\Delta(\S^3)\subset \S^3\times \S^3$,  the principal isotropy group of the suspension action on the base $\Sph^4$.
  \smallskip

$\bullet$ Fiber dimension $i=7$:

In analogy with our applications of \cite{uchida} and \cite{iwata}, we begin by analyzing the action on $B$, where $p: \Sph^{15}\to B$ is an equivariant fibration with fiber diffeomorphic to $\Sph^7$:

We claim that the $\G$ action (modulo kernel) on the homotopy $8$-sphere $B$ is equivariantly diffeomorphic to the spherical join action by  $\SO(2)\SO(7)$, or by $\SO(3)\SO(6)$, or to the suspension action by $\SO(8)$ on $\Sph^8$. Note that almost effectively, these are also actions by $\Spin(2)\Spin(7)$, $\Spin(3)\Spin(6) = \SU(2)\SU(4)$ and $\Spin(8)$, all subgroups of $\Spin(9)$.

\smallskip

To see this, note that from Table E in \cite{GWZ}, the multiplicity pair in $B$, coinciding with those in $\Sph^{15}$ are $(4,3)$ (typo in \cite{GWZ}), $(2,5)$, $(1,6)$, and $(1,7)$, for the potential irreducible representations by $\Sp(2)\Sp(2)$, $\SU(2)\SU(4)$ (or $\U(2)\SU(4)$), $\SO(2)\SO(8)$ and $\SO(2)\Spin(7)$ respectively. In addition, if the $\G$ representation is reducible the action on $B$ is necessarily a suspension action, and so the multiplicity pair is $(7,7)$. It follows that the most singular orbit in $B$ has dimension $0,1,2$, or $3$. Thus it is either a point, a circle or a sphere (in the latter cases since by transversality it is simply connected). Therefore, the dual  singular orbit is also a homotopy sphere (or a point) since it has the homotopy type of the complement of the orbit of codimension at least $3$, again, e.g., using transversality.  Because the singular orbits in $B$ are $\G$-homogeneous spaces, for dimension reasons it follows that, $\G$ can neither be $\Sp(2)\Sp(2)$ nor $\SO(2)\SO(8)$ (the former can not act nontrivially  on $\Sph ^3$, the latter can not act transitively on $\Sph^6$). Therefore, the singular orbits of the $\G$ actions on $B$ are respectively $(\Sph^2, \Sph^5)$, $(\Sph^1, \Sph^6)$ or two points $(p_{-}, p_+)$ corresponding to a representation of $\SU(2)\SU(4)$ (or its extension), $\SO(2)\Spin(7)$ or $\Spin(8)$ on $\Sph^{15}$. From isotropy groups data it follows that the only possible way in which these groups can act by cohomogeneity one on $B$ and in particular transitively on the respective pair of singular orbits is by the sum action of $\SO(3)\SO(6)$ and $\SO(2)\SO(7)$, and the suspension action by $\SO(8)$ respectively.

Next we want to prove that, the $\G$-representation on $\Sph^{15}$ must be one of the tensor representations of $\SO(2)\Spin(7)$, $\SU(2)\SU(4)$ or the reducible  $\Spin(8)$ representation  on $\Sph^{15}$. It remains to exclude the  tensor representation  of $\U(2)\SU(4)$. To do this note that if the $\U(2)\SU(4)$ representation descends via $p$ to $B$, then its center $\S^1$ must be in the kernel $\K_0$ of the $\G$-action on $B$ acting freely along the fibers of $p$. It follows that $p$ induces a fibration $\CP^7 \to B = \Sph^8$ with fiber $\Sph^7/\S^1 \cong \CP^3$. But such a fibration does not exist according to \cite{Ui}.

A set of compatible homomorphisms from the diagrams for the isotropy groups of the tensor representations $\SO(2)\Spin(7)$, $\SU(2)\SU(4)$ and $\Spin(8)$ on $\Sph^{15}$ to the reducible polar actions on $B$ is exhibited in Table \ref{Cayleydata} below (cf. \cite{FGT} for a correction for the isotropy groups data for $\SO(2)\Spin(7)$ case in \cite{GWZ}).

     {\setlength{\tabcolsep}{0.10cm}
\renewcommand{\arraystretch}{1.6}
\stepcounter{equation}
\begin{table}[!h]\label{exceptional}
      \begin{center}
       \begin{tabular}{|c||c|c|c|c|c|}

\hline
$\G$ & Representation   & $\K_-$  & $\K_+$   &  $\H$  &  $(\ell_-,\ell_+)$   \\

\hline\hline

       $\SO(2)\Spin (7)$   & $\Bbb R^2\otimes \Bbb R^8$     &    $\Delta(\SO(2))\SU(3)$   & $ \G_2  $ & $\SU(3)$ & $(1,6)$  \\
        \hline

$\SO(2)\Spin (7)$   & $\Bbb R^2\oplus \Bbb R^7$     &   $\SO(2)\SU(4)$  &  $ \Spin(7)$ & $\SU(4)=\Spin(6)$ & $(1,6)$  \\
        \hline

       $\SU(2)\SU(4)$     & $ \Bbb C^2\otimes \Bbb C^4$         &   $\Delta(\SU(2))\SU(2)$   & $\S^1\cdot \SU(3)$ & $\S^1\cdot \SU(2)$ & $(2,5)$   \\
       \hline

       $\SU(2)\SU(4)$     & $ \Bbb R^3\oplus \Bbb R^6$         &    $\Sp(1) \Sp(2)$    & $\S^1\SU(4)=\S^1\Spin(6)$  & $\S^1\cdot \Sp(2)=\S^1\Spin(5)$ & $(2,5)$   \\
       \hline

        $\Spin(8)$  & $ \Bbb R^8\oplus \Bbb R^8$  &  $\Spin(7)$  & $\Spin(7)$  & $\G_2$ & $(7,7)$ \\
        \hline

$\Spin(8)$  & $ \Bbb R^8\oplus \Bbb R^1$  &  $\Spin(8)$  & $\Spin(8)$  & $\Spin(7)$ & $(7,7)$ \\
        \hline

         \end{tabular}
      \end{center}
      \vspace{0.1cm}
         \caption{Fixed point isotropy representations of polar actions on $\Bbb {C}a\Bbb{P}^2$} \label{Cayleydata}
\end{table}}

  On the other hand, since $\Spin(3)\Spin(6)$, $\SO(2)\Spin(7)$ and $\Spin(8)$ are subgroups of $\Spin(9)$, it follows that they do act on the Cayley plane $\F_4/\Spin(9)$ in a polar fashion with isolated fixed points (cf. \cite{PTh}), hence act by isometries on the Hopf fibration $\Sph^{15}\to \Sph^8$. It is straigtforward to see that, for each of the $\G$-representations,  the set of compatible homomorphisms is unique up to conjugation. This proves the desired result.

\smallskip

$\bullet$ Cohomogeneity $k \ge 2$.

\smallskip
Whether or not the $\G$ action is irreducible, note that $p$ induces a $\G$ equivariant surjective map between
the chamber systems $\mathscr{C}(\Sph ^n,\G)
\to \mathscr{C}(B,\G)$ of the same type $\M$.
 Since $\mathscr{C}(\Sph
^n,\G)$ is a building it is both connected and simply connected. In
particular, $ \mathscr{C}(B,\G)$ is connected.

By our assumption on the slice representations it follows that $p$
yields an isomorphism between all proper residues of
$\mathscr{C}(\Sph ^n, \G)$ and $\mathscr{C}(B,\G)$. In particular,
$p: \mathscr{C}(\Sph ^n,\G) \to \mathscr{C}(B,\G)$ is a covering map
between chamber systems of type $\M$, and hence $\mathscr{C}(\Sph
^n,\G)$ is the universal cover of $ \mathscr{C}(B,\G)$.

By construction of the chamber topology of the universal cover
$\tilde{\mathscr{C}}(B,\G) = \mathscr{C}(\Sph ^n,\G)$ in the
previous section, it is apparent that it coincides with the topology
on $\mathscr{C}(\Sph ^n,\G)$ defined using the Hausdorff metric on
all compact subsets of $\Sph ^n$. The corresponding thick topologies
on $\mathscr{C}(B,\G)$ and $\mathscr{C}(\Sph ^n,\G)$ yield the
original topologies on $B$ and $\Sph ^n$ respectively. Moreover,
with this topology $\tilde{\mathscr{C}}(B,\G) = \mathscr{C}(\Sph
^n,\G)$ is a compact spherical building.

From Theorem \ref{comp group} we also know that the fundamental
group $\pi$ of the cover $\mathscr{C}(\Sph ^n,\G) \to
\mathscr{C}(B,\G)$ is a compact subgroup of  the topological
automorphism group ${\rm Aut}_{\rm top}(\mathscr{C}(\Sph^n,\G))$,
and that there is an action by $\tilde \G \subset {\rm Aut}_{\rm
top}(\mathscr{C}(\Sph ^n,\G))$ covering the $\bar {\G}$-action on $
{\mathscr{C}}$, where $\bar {\G}$ is $\G$ mod its kernel on $B$, and
$\tilde \G$ is an extension of $\bar \G$ by $\pi$. Moreover, the
actions by
 $\G\subset \tilde {\G}$ on $\Sph ^n$ are orbit equivalent, and $B$ is homeomorphic to $\Sph^n/\pi$.

Although in complete generality, we do not know much about the group
${\rm Aut}_{\rm top}(\mathscr{C}(\Sph ^n,\G))$, we claim that in our
case, $\pi \subset \tilde \G$ is
either $\S^1$ or $\S^3$ acting freely on $\Sph ^n$ by the Hopf
action.

\smallskip

Indeed, when the Coxeter diagram for $\M$ has no isolated nodes,
${\rm Aut}_{\rm top}(\mathscr{C}(\Sph ^n,\G))$ is a Lie group by
\cite{GKMW} (the rank is at least $3$). Moreover, since its maximal
compact subgroup acts linearly on $\Sph ^n$, the compact group $\pi$
acts linearly and freely on $\Sph ^n$,  hence $\pi$ is either
trivial, $\S^1$ or $\S^3$ acting on $\Sph ^n$ by the Hopf action.
Notice that $\G$ is either $\bar \G$ or $\tilde \G$ up to finite
kernel.

\smallskip

In general, note that each connected component of the diagram for $\M$ correspond to a $\G$ invariant linear subsphere, $\Sph_i$ of $\Sph^n$ on which $\G$ (mod its kernel) acts irreducibly in a polar fashion. Moreover, for each $i$, $\mathscr{C}(\Sph_i,\G)$ is a compact topological subbuilding of $\mathscr{C}(\Sph ^n,\G)$ invariant under $\tilde{\G}$ covering the chamber subsystem $\mathscr{C}(p(\Sph_i),\G)$.
Applying Lemma \ref{ext} we conclude that the compact group $\tilde{\G}$ is a Lie group that acts isometrically and is orbit equivalent to the action by $\G$ on $\Sph^n$. Thus also, $\pi$ is a compact Lie Group acting isometrically on $\Sph^n$ and the fibration $p: \Sph^n \to B$ is the orbit map by the free action of $\pi$.
\end{proof}

 We are now ready to prove the following

\begin{thm}\label{fixed}
Let $M$ be a simply connected compact positively curved polar $\G$
manifold. If $M^{\G} \ne \emptyset$ then $(M, \G)$ is equivariantly
diffeomorphic to an isometric, polar action of $\G$ on a compact
rank one symmetric space.
\end{thm}

\begin{proof}
We will see in particular that $M$ is a sphere if and only if the
section $\Sigma$ is a sphere.

Let us first deal with the case where

\no $\bullet$ $\Sigma$ is a $k$-sphere, $k \ge 2$:

Recall by Proposition \ref{fix} that in this case $M^{\G} = \Sigma^{\W} =: \Sph \subset \Sigma = \Sph * \Sph^{\ell}$, and $C = M/\G = \Sigma/\W = \Sph * \Delta$, where $\Delta = \Sph^{\ell}/\W$ and $\dim \Delta = \ell \ge 1$.

The smooth spherical join description $\Sigma = \Sph *
\Sph^{\ell}$ yields a decomposition of $\Sigma$ as a union
of tubular neighborhoods of $\Sph$ and of $\Sph^{\ell}$. Applying $\G$ gives
a smooth decomposition of $M$ into a union of tubular
neighborhoods of $\Sph = M^{\G}$ and the $\G$ invariant manifold $\G \Sph^{\ell} =: \Sph' \subset M$. (In the metric chosen note that the cut locus of $\Sph$ in $M$
is $\Sph'$ and vice versa, at distance $\pi/2$ from one another).
Note that $\Sph'$ is a
polar $\G$ manifold with section $\Sph^{\ell}$, polar group $\W$ and $\Sph'/\G = \Sph^{\ell}/\W = \Delta$. Moreover, if $\Sph^0 = \{p_-\cup
p_+\} \subset \Sph$ is a pair of antipodal points in $ \Sph$, we see
that $\G \Sph^{\ell + 1} = \G ( \{p_-\cup p_+\} * \Sph^{\ell})$ is a
$\G$ invariant polar submanifold $N \subset M$ with two isolated
fixed points $p_{\pm}$, section $\Sigma_0 := \{p_-\cup p_+\} * \Sph^{\ell} \subset \Sigma$ and polar
group $\W$. From this it in particular follows that $\Sph'$ is
equivariantly diffeomorphic to the unit sphere $\Sph^{\perp}$ at a fixed point, say $p_-$ of $\G$ in $N$, and that $N$ is
equivariantly diffeomorphic to the suspension of this. Of course
$\Sph^{\perp}$ is the normal sphere of $\Sph = M^{\G}$ in $M$ at $p_-$, and
a similar argument now shows that $M$ is equivariantly diffeomorphic
to  $\Sph * \Sph^{\perp}$ where $\G$ acts trivially
on  $\Sph$ and by the isotropy representation of $\G$ on the normal
sphere $\Sph^{\perp}$.

We now turn to the case where

\no $\bullet$ $\Sigma$ is a projective $k$-space, $k \ge 2$:

Since $\Sigma^*=\Sigma/\W$ is a spherical $k$-simplex $\Delta$ and $M^{\G} \subset
\Sigma^{\W}$, we know that $M^{\G}$  is contained in the vertices of any chamber $C = \Delta=M/\G = \Sigma/\W$ by Proposition
\ref{chamber group}.

Let $p_0 \in M^{\G}$ be such an isolated fixed point, and $\Delta_0$ the chamber face opposite $p_0$ in a fixed chamber $C= \Delta$, i.e., $C = \Delta = p_0 * \Delta_0$. It follows that
$B:= \G \Delta_0 \subset M$ is a polar space form $\G$ submanifold of
$M$ with section $ \RP^{k-1} \subset \Sigma = \RP^k$ and polar group induced from
$\W$.  Arguing as above,
$M$ is the union of a ball centered at $p_0$ and a tubular
neighborhood of
 $B$. (In the chosen metric $B$ is the cut locus of $p_0$ and vice versa at distance $\frac \pi 2$). In particular, we have an equivariant sphere fiber bundle
$p: \Sph \to B$ (with nontrivial fiber) between polar $\G$ manifolds
 with the same orbit space $\Delta_0$, where $\Sph$ is the unit sphere at $p_0$. Note also, that for any $0 < r <  \frac \pi 2$ the metric $r$-sphere $\Bbb S(p_0,r)$ centered at $p_0$ is a polar $\G$ manifold equivariantly diffeomorphic to $\Sph$ via scaling and $\exp_{p_0}$. Moreover, $\Bbb S(p_0,r)$ coincides with the boundary $\Bbb S(B; \frac \pi 2 -r)$ of the $ \frac \pi 2-r$ tube $\Bbb D(B;  \frac \pi 2 -r)$ of $B$, and in this way $p$ can also be viewed as the projection from the unit normal sphere bundle, $\Sph^{\perp}(B)$ of $B$ to $B$. By transversality we see that $B$ is simply connected. From this description it follows that $M$ is equivariantly diffeomorphic to a projective space once it is established that $p$ is a Hopf fibration. To see this it remains to check the assumption on the slice representations in Lemma \ref{hopf}.

 Let $\gamma$ be a geodesic in $C$ from $p_0$ to $q\in \Delta_0$ perpendicular to $\Delta_0$. By the slice theorem,  $\G\times _{\G_q} \hat V_q$, is a $\G$-equivariant tubular neighborhood of the orbit $\G(q)$,  where $\hat V_q$ is the slice in $M$. Since $B$ is $\G$-invariant, we get
a $\G_q$-invariant decomposition $\hat V_q=V_q\oplus V_q^\perp$ where $V_q$ is the slice in $B$, and $V_q^\perp$ is the normal space to $B$ at $q$.  Note that, from the slice representation of $\G_q$ on $V_q\oplus V_q^\perp$, the slice $\hat V_x$ for $\G_x$ at $x\in \gamma $ different from $q$ is naturally identified with $V_q\oplus T_x \gamma $. Therefore, the slice of the orbit at $x$ inside $\Sph (B; \frac \pi 2-r)$ is canonically identified with $V_q$. Moreover, the orbit space $V_q/\G_x=V_q/\G_q$, a cone over the space of directions at $q\in \Delta_0$.  The desired result follows.
 \end{proof}

\section{Fixed point Free Reducible Actions} \label{genRed}     

In all remaining cases, the orbit space $M^* = \Sigma^*$ is a
simplex $\Delta$ isometric to all chambers in $M$. Moreover,
$\Delta$ is a spherical join $\Delta = \Delta_- * \Delta_+ =
\Delta^{m_-} * \Delta^{m_+}$, corresponding to two \emph{dual} $\W$
invariant subsections $\Sigma _-$ and $\Sigma _+$, where $\Sigma
_\pm=\Bbb S^{m_\pm}$ or the projective spaces $\RP^{m_\pm}$.

Viewing $\Delta$ also as a subset of a fixed section $\Sigma$,
clearly $B_- = \G \Delta^{m_-}$ and $B_+ = \G \Delta^{m_+}$ are two
polar $\G$ submanifolds in $M$ with sections $\Sigma_{-},
\Sigma_{+}\subset \Sigma$ and Weyl groups $\W$ (mod kernel).  In
particular, $B_{\pm}$ are polar space forms of spherical type.
Moreover, just like $\Delta$ can be viewed as the union of two
tubular neighborhoods of the $\Delta^{m_\pm}$, $M$ is the union of
tubular neighborhoods of the $\G$ submanifolds $B_{\pm}$.

In the remaining cases where no fixed points are present our first goal is to exhibit a geometric description of $M$ as a projective space in which $B_{\pm}$ is a dual pair of projective subspaces. The pivotal steps are to show that these pairs are the cut loci of one another, and that for each point $p_{\pm} \in B_{\pm}$, the exponential map (up to scaling) from the unit normal sphere at  $p_{\pm}$ to $B_{\pm}$ defines a map to $B_{\mp}$ which in turn is a Hopf fibration. This is in spirit achieved by reducing it to the fixed point case where the groups in question are the isotropy groups at $p_{\pm}$. Analyzing and making full use of equivariant restrictions forced by this description will then yield a proof of our main result in this section:

\begin{thm}[Non-fixed point] \label{nonexcep}
A reducible fixed point free polar action on a simply connected
positively curved manifold is equivariantly diffeomorphic to an
isometric polar action on a rank one symmetric space, excluding the
Cayley plane.
\end{thm}

 Note, that when say $m_{+} \ge 1$, the slice representation at each
vertex of $\Delta_{+} \subset \Delta$ is reducible. In particular,
all vertex representations are reducible except possibly the one
corresponding to say $\Delta_-$ when it is a point.

\smallskip

The following is a key step based on the primitivity lemma
\ref{prim}

\begin{lem}[Dual Generation]\label{dual generation}
For any regular pair $p_{\pm} \in B_{\pm}$, the action of the
isotropy groups $\G_{p_{\pm}}$ restricted to $B_{\mp}$ is orbit
equivalent to the action of $\G$ restricted to $B_{\mp}$.
\end{lem}

\begin{proof}
By the primitivity theorem,  $\G_{p_-}$ is generated by the face
isotropy groups, $\G_{v_1}, \ldots, \G_{v_{m{_+}+1}}$ of the faces,
$\Delta{_-} * \Delta^{{m_+} -1}_{v_i}$ containing $\Delta_-$, and
similarly $\G_{p_+}$ is generated by the remaining face isotropy
groups, $\G_{u_1}, \ldots, \G_{u_{m_{-}+1}}$, namely of the faces,
$\Delta^{m_{-}-1}_{u_j} * \Delta_+$ containing $\Delta_+$. Note that
any face containing $\Delta_-$ is perpendicular to any face
containing $\Delta_+$. In particular, if $\G_{v,u}$ is the isotropy
group at an intersection point of two such faces with isotropy
groups $\G_v$
 and $\G_u$, the slice representation of $\G_{v,u}$ restricted to the normal sphere of its fixed point set is a reducible cohomogeneity one action with
 singular isotropy groups $\G_v$ and $\G_u$. As a
  special case of the primitivity theorem we already know that $\G_v$ and $\G_u$ generate $\G_{v,u}$.
  However, since the action is reducible we have that actually $\G_v \G_u = \G_u \G_v = \G_{v,u}$ as
  sets. Notice that this is equivalent to the fact that in the slice representation of $\G_{v,u}$, the isotropy group $\G_v$ is transitive
  on the opposite singular orbit and vice versa.

We now claim that  $\G = \G_{p_-} \G_{p_+}$. From the primitivity
lemma we know that any $g \in \G$ can be written as a word of
elements from $\G_{v_1}, \ldots, \G_{v_{m_{+}+1}}, \G_{u_1}, \ldots,
\G_{u_{m_{-}+1}}$. Using that $\G_{v_i} \G_{u_j} = \G_{u_j} \G_{v
_i}$ for all $i = 1, \ldots, m_{+}+1$ and
 $j = 1, \ldots, m_{-}+1$ we can rewrite any such word also as a word in the $\G_v$'s times a word in the $\G_u$'s, i.e., $\G = \G_{p_-} \G_{p_+}$.
 The same reasoning
 shows that $\G = \G_{p_+} \G_{p_-}$, and hence completes the proof of the lemma.
\end{proof}

 The above lemma will allow us to use the input from the previous
section. For this we let  $\Gamma(p_{\pm})$ be the set
consisting of all minimal geodesics from
regular points $p_{\pm}$ to $B_{\mp}$. In addition to viewing this as a set of geodesics, we will also view it as a subset of $M$ whose points are the points of all those geodesics. As such it can also be described as $\Gamma(p_\pm)=\G_{p_{\pm}}(p_\pm*\Delta_\mp) \subset M$. Note that $\Gamma(p_{-}) \cap
\Gamma(p_{+})$ is the set of all minimal geodesics joining $p_-$ and
$p_+$. Since $\G_{p_\pm}$ is independent of $p_{\pm}$,
we will use the notation  $\G_\pm$ instead.

It will also be useful to let  $\Gamma(p_{\pm})(r)$ denote the
subset  of $\Gamma(p_{\pm})$ at distance $r$ from $p_{\pm}$, and to
let  $\hat{\Gamma}(p_{\pm})$ denote the negative of the terminal
directions of the geodesics in $\Gamma(p_{\pm})$.

\begin{rem}
The following are immediate consequences of the Dual Generation
lemma \ref{dual generation}, and the decomposition of a section
$\Sigma \supset \Sigma_{\pm}$ corresponding to $\Delta = \Delta_- *
\Delta_+$.
\begin{itemize}
\item
The cut locus $C(B_{\pm}) = B_{\mp}$, and $B_{\pm}$ are at distance
$\pi/2$ from one another.
\item
$\Gamma(p_{\pm}) - B_{\mp}$ is a smooth submanifold of $M$
diffeomorphic to the open $\pi/2$ ball in the normal space, $T^{\perp}_{\pm}$ to $B_{\pm}$ at $p_{\pm}$
via the
exponential map.
\item
The map $\gamma_{p{\pm}}: \Sph^{\perp}_{p_{\pm}} \to B_{\mp}$ taking
a unit vector to the corresponding geodesic at time $\pi/2$ is
smooth, $\G_{\pm}$ equivariant and takes chambers to chambers.
 \end{itemize}

\end{rem}

Moreover, we have

\begin{lem}\label{fibration}
The map $\gamma_{p{\pm}}$ is a $\G_{\pm}$ equivariant
Hopf fibration.
\end{lem}

\begin{proof}
For the sake of simplicity we will use $\bar x$ to
denote the image $\gamma_{p{\pm}}( x)$ of $x\in
\Sph^{\perp}_{p_{\pm}}$. Since $\gamma_{p{\pm}}$ is a smooth $\G_{\pm}$ equivariant map that takes chambers to chambers and with the same orbit space, by the Hopf lemma \ref{hopf} it remains to verify that the slice representations of $\G_{\pm, v}$ and
$\G_{\pm, \bar v}$ are orbit equivalent. For this it suffices to show that the dimensions of the corresponding slices agree, or equivalently that corresponding principal orbits of isotropy groups have the same dimension. So let $x$ be a
point of  principal orbit type of the slice representation of
$\G_{\pm, v}$. We claim that $\dim(\G_{\pm, v}(x))=\dim(\G_{\pm, \bar v}(\bar
x))$. Note that $\G_{\pm, \bar v}=\G_{\pm, \bar x}\G_{\pm, v}=\G_{\pm,v}\G_{\pm, \bar x}$ where the latter follows from the Dual Generation lemma \ref{dual generation}. Therefore, the orbit $\G_{\pm, \bar v}(\bar
x)=\G_{\pm, \bar v}/\G_{\pm, \bar x}=\G_{\pm,  v}\G_{\pm, \bar
x}/\G_{\pm, \bar x}$. In particular, $\G_{\pm,  v}$ acts
transitively on this orbit with isotropy group $\G_{\pm,  v}\cap
\G_{\pm, \bar x}$. However, since clearly $\G_{\pm,  v}\cap \G_{\pm, \bar
x}= \G_{\pm,  x}$, this completes the proof.
\end{proof}

 The following plays a pivitol role in the geometric and equivariant description of $M$:

\begin{lem} [Reduction]\label{reduction}
For all regular $p_{\pm} \in B_{\pm}$, $\Gamma(p_{\pm})$ are
$\G_{\pm}$ invariant submanifolds of $M$. Moreover,

$\bullet$ $\Gamma(p_{\pm})$ is $\G_{\pm}$ equivariantly
diffeomorphic to $\Disc^{\perp}_{p_{\pm}}$

 \no if the section is a sphere, and

$ \bullet$ $\Gamma(p_{\pm})$ is $\G_{\pm}$ equivariantly
diffeomorphic to a complex or quaternionic projective space

\no if  the section is a projective space.
\end{lem}

\begin{proof}  The key issue is to see that $\Gamma(p_{\pm})$ are submanifolds as claimed. From the remark above this is clear except along $B_{\mp} \subset \Gamma(p_{\pm})$.

 From the Hopf Lemma \ref{hopf} and Lemma \ref{fibration} above we know that  $\gamma_{p_{\pm}}: \Sph_{p_{\pm}}^{\perp} \to B_{\mp}$ is either a diffeomorphism or a Hopf map.
 Clearly, $\Gamma(p_-) \cap \Gamma(p_+)$ is in bijective correspondence with the fiber of $\gamma_{p_-}$ over $p_+$ and the fiber of $\gamma_{p_+}$ over $p_-$, when viewing it as the set of minimal geodesics between $p_-$ and $p_+$. In particular,
 both maps are of the same type, corresponding to $\Gamma(p_-) \cap \Gamma(p_+)$ being either one geodesic, an $\Sph^1$, an $\Sph^3$ or an $\Sph^7$ family of geodesics. Moreover, this description also shows that the linear span of the initial vectors of the geodesics in $\Gamma(p_-) \cap \Gamma(p_+)$ at both $p_-$ and $p_+$ are linear subspaces of the corresponding  normal spaces to $B_{\mp}$.

 Now consider the initial vectors of the geodesics in $\Gamma(p_-)$ starting at $B_+$. This subset $\hat{\Gamma}(p_{-})$ of the unit normal bundle $T_1^{\perp}B_+$ of $B_+$ is canonically
 a smooth submanifold diffeomorphic to $\Sph_{p_-}^{\perp}$ via say $\Gamma(p_{-})(\frac{\pi}{2} -1)$. In particular, it is a smooth section of the unit normal bundle to $B_+$ when each $\Gamma(p_-) \cap \Gamma(p_+)$ is just one geodesic, or equivalently $\gamma_{p_-}: \Sph^{\perp}_{p_-} \to B_+$ is a diffeomorphism. In the other cases it is the unit sphere bundle of a smooth linear subbundle of the normal bundle to $B_+$ and $\gamma_{p_-}: \Sph^{\perp}_{p_-} \to B_+$ is a Hopf fibration. By equivariance then clearly each $\Gamma(p_{\mp})$ is $\G_{\mp}$ equivariantly diffeomorphic to a complex or quaternionic space, or to the Cayley plane if $\Gamma(p_-) \cap \Gamma(p_+)$ is an $\Sph^7$ family of geodesics. Since by assumption neither $B_{\pm}$ is a point the latter case does not appear. Indeed, if so both $B_{\pm}$ would be homotopy
8-spheres. Moreover, from the geometric decomposition and Poincar\'{e} duality it follows that $M$ would be a 24-dimensional manifold with integral cohomology algebra a truncated polynomial algebra with generator in degree $8$. This contradicts a
well-known topological theorem (cf. [Ha] page 498, Corollary 4).

\end{proof}

\smallskip

Having dealt with all cases where the Diagram for the Coxeter matrix
has no isolated nodes, and where the action has fixed points, we
assume from now on that $B_-$ is an orbit corresponding to an
isolated node of the diagram. Thus, we will use a decomposition
\begin{center}
$\Delta = \Delta_- * \Delta_+$, where $\Delta_- = \Delta^0 = p_-$ is
a vertex,
\end{center}
\no corresponding to an isolated node, and $\Delta_+$ corresponds to
the rest of the diagram. In particular,
\begin{center}
$\G_+$ acts transitively on $B_-$, as well as on each normal sphere
$\Sph_{+} ^{\perp}$ to $B_+$ along $\Delta_+$.
\end{center}

\bigskip

We are now ready to complete the proof of Theorem \ref{nonexcep}:

\begin{proof} [Proof of Theorem \ref{nonexcep}] We first consider the case where the section $\Sigma = \Sigma_- * \Sigma_+$ is a
sphere:

\no  By Lemma \ref{reduction} the $\G_{\mp}$ action on $B_{\pm}$
is equivariantly equivalent to the slice representation on the
normal sphere $\Sph^{\perp}_{p_{\mp}}$ which we will denote by $\Sph(V_{\pm}) = \Sph_{\pm}$.  Note that
$\G$ as well as $\G_+$ acts transitively on $B_-$ since $\Delta _-$ is a point. In particular $\G$ acts linearly on $\Sph_-$ identified with $B_-$. If also the $\G$ action on $B_+$ when identified with $\Sph_+$ is linear, we claim that the induced sum action on $\Sph(V_- \oplus V_+)$ is equivalently diffeomorphic to the $\G$ action on $M$. To see this, choose $p_-^* \in \Sph_-$ with $\G_{p_-^*} =  \G_{p_-}= \G_-$ and a $\G_-$ equivariant diffeomorphism from $\Gamma
(p_-)$ to the join ${p_-^*} * \Sph_+\subset \Sph_-*\Sph_+$. This extends to a well defined $\G$ equivariant diffeomorphism from $M$ to $\Sph_-*\Sph_+$ by invariance. The proof is completed now by Lemma \ref{ext}, since the polar $\G$-action on $B_+$ is linear orbit equivalent to the $\G_-$ action.

Now suppose $\Sigma$ is a projective space:

\no By Lemma \ref{reduction} the Cayley plane cannot appear, and moreover $B_\pm$ are both projective spaces over $\Bbb F$, with $\Bbb F=\C$ or $\Bbb H$. In fact, for each regular $p_{\pm} \in B_{\pm}$, the slice representation of $\G_{\pm}$ restricted to the normal space $V_{\pm}$ at $p_{\pm}$ to $B_{\pm}$ preserve an $\Bbb F$ structure and descends to a polar action on $B_{\mp}$ orbit equivalent to the restriction of $\G$ to $B_{\mp}$. This will guide us to the construction of a representation of $\G$ (or frequently an extension of it) on $V_+ \oplus V_-$ preserving an $\Bbb F$ structure, which together with the scalar multiplication of $\Bbb F^*$ is polar, such that the induced $\G$ action on the projective space
$\Bbb {FP}(V_+\oplus V_-)$ is equivariantly equivalent to that of $\G$ on $M$.

We divide the proof into two cases corresponding to (a) $\dim B_- > 2$ and (b) $\dim B_- = 2$ (noting that  $\dim B_- < 2$ is covered by the fixed point case).

Throughout $\K_{\pm} \lhd \G$ will denote the identity component of the kernel of the $\G$-action restricted to $B_{\pm}$.

In case (a), we make the following claim: There is a normal subgroup $\H \lhd \K_+$ acting transitively on the normal spheres to $B_+$ such that $\K_+=\H\cdot \K_0$, where $\K_0 = \1, \S^1$ or $\S^3$, is the identity component of the kernel of the action by $\K_+$ on $B_-$. We will see later on that $\K_0=\1$  when $\Bbb F=\Bbb H$, and $\K_0=\1$ or $\S^1$ when $\Bbb F=\Bbb C$.

To prove the claim note that $\dim \Delta_+ \ge 1$. Choose a pair of vertices  $v_1, v_2\in \Delta_+$ and consider the slice representations at the vertices.  By Lemma \ref{general} below  it follows, by assumption on dim $B_-$, that up to a finite cover, there is a normal subgroup of rank at least $2$ (of simple type), say $\H_{v_i}\lhd \G_{v_i}$ for $i=1, 2$, acting transitively on the normal sphere to $B_+$ at $v_i$ but trivially on
the slice tangent to $B_+$. Clearly each $\H_{v_i}$  is also a normal (simple) subgroup in the principal isotropy group $\G_+$ of the $\G$ action on $B_+$. Since $\G_+$ (modulo kernel) has a unique normal simple subgroup (of rank at least 2) acting transitively on the normal sphere, it follows that  $\H_{v_1}=\H_{v_2}\lhd \G_+$, and will be denoted by $\H $. By primitivity
$\G=\langle \G_{v_1},\G_{v_2}\rangle $, and thus $\H$ is a normal subgroup in $\G$. Therefore $B_+$ is fixed pointwise by $\H$. In particular,
$\H \lhd \K_+ $ acts transitively on $B_-$ with kernel $\K_0$, and thus $\K_+=\H\cdot \K_0$.

Since $\K_+\lhd \G$, we can write $\G=\K_+\cdot \L=\H\cdot \K_0\cdot \L$ where $\L\lhd \G$ is a connected normal subgroup which clearly acts  almost effectively on $B_+$ in a polar fashion.    Since $B_-$ is a projective space over $\Bbb C$ or $\Bbb H$ of dimension at least $4$, and therefor any almost effective transitive action on it is the  linear action by $\SU(k+1)$ or $\Sp(k)$ for some $k\geq 2$ (cf. \cite{On}, pp. 264-5), we conclude that $\H=\SU(k+1)$ or $\Sp(k)$. Thus for the kernel of the $\G =\H\cdot \K_0\cdot \L$ action on $B_-$, we get $\K_-=\K_0\cdot\L$. In particular, $\K_0=\K_-\cap \K_+$ fixes both $\B_\pm$ pointwise, and acts almost effectively on the normal spheres to $B_\pm$ (when non trivial), preserving all $\Gamma(p_-) \cap \Gamma(p_+)$.

In summary the ``face" isotropy groups $\G_\pm$ and their kernels $\K_\pm$ can be read off from the following group diagrams, where the vertical inclusions are left out:

\begin{equation}\label{Ldiagram}
\begin{split}
\xymatrix{
& \G =\H\cdot \K_0\cdot \L\ar@{<-}[dr]^{ } \ar@{<-}[dl]_{ } & \\
\G_{-}=\H_-\cdot \K_0\cdot \L \ar@{<-}[dr] & & \G_+=\H\cdot \K_0\cdot \L_+  \ar@{<-}[dl] \\
& \G_-\cap \G_+= \H_-\cdot \K_0\cdot \L_+ &\\
\K_-= \K_0\cdot \L \ar@{<-}[dr] & & \K_+= \H\cdot \K_0  \ar@{<-}[dl] \\
& \K_-\cap \K_+= \K_0 & }
\end{split}
\end{equation}
\noindent Here  $\H _-$, respectively $\L_+$, is the principal isotropy group of $\H$ on $B_-$ and of $\L$ on $B_+$ respectively. Thus corresponding to the possible $\H$ above, we have $\H_-=\U(k)$,  $\Sp(k-1)\S^1$ or $\Sp(k-1)\Sp(1)$. Let  $\H_0\lhd \H_-$ denote the normal factor $\S^1$ or $\Sp(1)$ of $\H_-$. Thus, $\H_0 = \S^1$ corresponds to $\Bbb F = \Bbb C$ and $\H_0 = \Sp(1) $ to $\Bbb F = \Bbb H$. For this we have:

\smallskip
$\bullet$ $\H_0$ acts freely on $\Sph(V_-)$ along the Hopf fibers.

\smallskip
To see this, consider  for any fixed $p_+ \in B_+$, the Hopf map $\gamma_{p_+}: \Sph^{\perp} _{p_+} \to B_-$. From the transitive actions by $\H$ on $\Sph^{\perp}_{p_+}$ descending to $B_-$ we know that $\H_0$ acts freely along the fibers of $\gamma_{p_+}$. These fibers for $p_- \in B_-$ are in one to one correspondence with $\Gamma(p_-)\cap \Gamma(p_+)$, $p_- \in B_-$. Turning things around, these are also in one to one correspondence with the fibers of the Hopf map $\gamma_{p_-}: \Sph^{\perp} _{p_-} \to B_+$ where now $p_-$ is fixed. This proves the assertion and has the following consequence:

$\bullet$
$\K_0$ is trivial when  $\Bbb F = \Bbb H$ .

Indeed, since both $\K_0$ and $\H_0= \Sp(1)$ act (almost) effectively on $\Sph(V_-)$ this follows from Lemma 7.9 below, because the slice representation of $\G_-$ on $V_-$ descends to a fixed point free action on $B_+$ with $\H_0 \cdot \K_0\lhd \H_-\cdot \K_0$ in its kernel.
\smallskip


We now proceed to set up a projective model, $\Bbb {FP}(V_+\oplus V_-)$ with a linear polar $\G$-action with the field $\Bbb F=\Bbb C, \Bbb H$ as indicated in our strategy above:

Consider the product representation on $V_+\oplus V_-$ by $\K_+ \times \K_-=\H\cdot \K_0 \times \K_0\cdot \L$ which on each summand preserves an $\Bbb F$ structure, i.e., descends to a polar action on $\Bbb {FP}(V_+)$, respectively on $\Bbb {FP}(V_-)$. When $\K_0$ is trivial, obviously the sum  $\Bbb F$ structure is preserved as well, and the  $\K_+ \times \K_- = \H \times \L$ action descends to a polar action on $\Bbb {FP}(V_+\oplus V_-)$. When $\K_0 = \S^1$ and hence $\Bbb F=\Bbb C$, the action by the diagonal $\S^1 = \Delta(\K_0) \lhd \K_0\times \K_0$ defines a $\Bbb C$ structure on the sum preserved by $\K_+ \times \K_-$ descending to a polar action by $\G =\H\cdot \K_0\cdot \L$ on $\Bbb {FP}(V_+\oplus V_-)$. These are the models.

Given such a model, fix a point $p_-\in B_-$, and choose a point $p^*_-\in \Bbb {FP}(V_+)$ so that $\G_{p^*_-}=\G_{p_-}$. By Lemma 7.5   $\Gamma (p_-)$ is $\G_{p_-}$-equivariantly diffeomorphic to the linear projective subspace of $\Bbb {FP}(V_+\oplus V_-)$ containing $p^*_-$ and $\Bbb {FP}(V_-)$. As in the spherical section case this extends to a well defined $\G$ equivariant diffeomorphism from $M$ to $\Bbb {FP}(V_+\oplus V_-)$ by invariance, and we are done.
\smallskip

Finally, let us consider the only remaining case (b) where  $B_-=\C\Bbb
P^1$ and $B_+$ is a complex projective space of real dimension at
least $4$, since the cases of dim $B_+\leq 2 $ reduce to the fixed point case, because dim $\Delta^+\geq 1$.

Since $\G$ acts transitively on $B_- = \CP^1 = \Sph^2$ we can write, $\G=\K_- \cdot \S^3$,  with $\K_-$ the
kernel of the action on $B_-$. Hence, $\G_-=\K_- \cdot \S^1$,
where $\S^1\subset \S^3$. By Lemma \ref{dual generation} the subaction of $\G_-$ is orbit equivalent to the $\G$-action on $B_+$. It follows from Lemma \ref{kernels} below that the factor $\S^3 $ acts  trivially on $B_+$, that is $\S^3\lhd \K_+$, the kernel of the $\G$  action on $B_+$. Therefore, the slice subrepresentation of $\K_+\vartriangleright \S^3$ on $V_+\cong \Bbb C^2$ descends to the transitive action on $B_-$. This implies that $\K_+=\S^3$ or $\U(2)$, and accordingly $\K_0=\K_+\cap \K_-=\1$ or $\S^1$. We are now in a situation similar to case (a) with $\Bbb F=\Bbb C$, and hence $M$ is
$\G$-equivariantly diffeomorphic to the complex projective space $\Bbb P(V_+ \oplus V_-)$. This completes the proof.
\end{proof}

In conclusion, here are the facts we used about representations in the proof of Theorem \ref{nonexcep}:

\begin{lem}\label{general}
\label{general} Let $\rho: \G\to \SO(V_0\oplus \cdots \oplus V_k)$, be a reducible polar representation, where the $V_i$ are irreducible $\G$-modules ($k\geq 1$) . Suppose the $\G$-action on the unit sphere $\Bbb S(V_0\oplus \cdots \oplus V_k)$ descends to a polar action on the projective space $\Bbb {FP}(V_0\oplus \cdots \oplus V_k)$ where $\Bbb F=\Bbb C$ or $\Bbb H$.  If $\dim V_0\geq 5$ and $\dim V_0/\G = 1$, then there is a normal simple subgroup $\H\lhd \G$ of rank at least $2$ acting transitively on the unit sphere $\Bbb S(V_0)$ but trivially on $V_1\oplus \cdots \oplus V_k$.
\end{lem}

\begin{proof}
Since $\G$ is transitive on $\Sph(V_0)$, by the list of transitive actions on the spheres it follows that, $\G=\H\cdot \G'$, where $\H$ is a simple normal subgroup of $\G$ acting transitively on $\Sph(V_0)$, with principal isotropy group $\H_0$. For dimension reason, $\text{rank } \H\geq 2$. Moreover, $\H$ is a special unitary group or a sympletic group, since it acts transitively on the projective space $\Bbb{FP} (V_0)$.

We argue by contradiction. Assume $\H$ acts nontrivially on $V_i$ for some $i\geq 1$. For the restricted reducible polar representation of $\G$ on $V_0\oplus V_i$, since the principal isotropy group of $\G$ on $V_0$ is  $\H_0\cdot \G'$, by \cite{Be} Theorem 2   it follows that $\G$ is orbit equivalent to $\H_0\cdot\G'$ on $V_i$,  hence, by \cite{Da}, orbit equivalent to the isotropy representation of an irreducible symmetric space. Note that $\H_0\varsubsetneq \H$ is not a normal subgroup, and by assumption $\H$ is nontrivial on $V_i$. By the list in \cite{EH} it follows that,  if dim $V_i/\G\geq 2$, then $\H$ is $\Spin(8)$ or $\Spin(7)$ which is neither a unitary nor a symplectic group,  contradiction.

It remains to consider the case of dim $V_i/\G=1$, i.e., $\G$ acts transitively  on $\Bbb S(V_i)$. By the assumption that $\H$
acts non-trivially on $V_i$,  from the list of transitive actions on spheres it follows  that:

\smallskip
$\bullet$ The rank at least $2$ simple group $\H$ acts transitively on $\Sph(V_i)$. Furthermore, by dividing the kernel of the $\G'$-action, the transitive action of $\H\cdot \G'$ on $\Sph(V_i)$ reduces to an almost effective action of $\H\cdot \K_0$ where $\text{rank } \K_0\leq 1$.
\smallskip

\no Note that $\text{dim } \Sph(V_i )\geq 4$, since a rank at least $2$ simple group can not act nontrivially on a lower dimensional spheres. Recall that $\H_0\cdot \G'$ is orbit equivalent to the $\H\cdot \G'$-action on $V_i$, i.e., transitive on $\Sph(V_i)$.  For dimension reasons it follows that the $\H_0$ action is also transitive on $\Sph(V_i)$ since the $\G'$-action modulo kernel reduces to the $\K_0$-action where $\text{rank } \K_0\leq 1$.

  In summery, $\H$ acts transitively on $\Sph(V_i)$ orbit equivalent to the subaction of  $\H_0\varsubsetneq \H$ on $\Sph(V_i)$ where $\H/\H_0\cong\Sph(V_0)$. Recall that $\H$ is a special unitary group or a symplectic group. From the list of transitive actions on spheres it follows that $\H=\SU(4)$ acting on
$V_0\oplus V_i=\Bbb R^8\oplus \Bbb R^6$, via the standard complex representation and the real representation of $\SU(4)=\Spin(6)$ on the summands (cf. \cite{Be}). Since the $\Spin(6)$ action on $\Bbb R^6$ does not descend to the complex projective plane, a contradiction is reached.
\end{proof}

In the proof of the following result we will freely use the language and results about chamber systems and their universal covers developed in section 4, translated to the current setting.

\begin{lem}\label{kernels} Let $\rho: \G\to \SO(V)$ be an almost effective $\Bbb C$-linear polar  representation, descending  to a polar action on the projective space $\Bbb {CP}(V)$.

Suppose $\G= \K\cdot \S^1\subset \G'= \K\cdot \S^3$, and assume moreover that  $\G'$ acts on the projective space extending the $\G$-action with the same orbits. Then  the factor $\S^3$ of $\G'$ is in the kernel of the action on $\Bbb {CP}(V)$.
\end{lem}

\begin{proof} We first prove that the factor $\S^3$ action is trivial on $\Bbb {CP}(V_i)$ for each irreducible summand $V_i\subset V$. This is clear for any rank $1$ summand (summand $V_i$ with $\dim V_i/\G = 1$), since   $\G'\supset \G$ acts transitively on $\Bbb {CP}(V_i)$, hence modulo kernel it is either a unitary group or a symplectic group. If the rank of a summand is $2$, it follows immediately from [Uc].
If the rank of a summand is at least $3$, by
section 4,   the
action of  $\G'\supset \G$ lifts to a polar representation by an extension by the deck transformation group $\S^1$, i.e., $\S^1\cdot \G'$,
acting on the universal cover, i.e., $\Sph(V_i)$, orbit equivalent to the $\S^1\cdot\G$-representation.
By the classification of [EH] it follows that, the $\S^3$ factor as well as the factor $\S^1\subset \S^3$ must act trivially on $\Sph(V_i)$. In general, suppose $\S^3$
acts trivially on both $\Bbb {CP}(V_i)$ and $\Bbb {CP}(V_j)$, then $\S^3$ acts on $\Bbb {CP}(V_i\oplus V_j)$ with fixed point set the disjoint union $\Bbb {CP}(V_i) \bigsqcup\Bbb {CP}(V_j)$. Since  the space of geodesics $\Gamma (p_i) \cap \Gamma(p_j)$ joining $p_i\in \Bbb {CP}(V_i)$ and $p_j\in \Bbb {CP}(V_j)$ is an $\Sph^1$ family of geodesics, $\S^3$ must act trivially on that as well for all
$p_i, p_j$. Thus, $\S^3$ acts trivially on  $\Bbb {CP}(V_i\oplus V_j)$. By induction it follows that the $\S^3$ action is trivial on $\Bbb {CP}(V)$.
\end{proof}

The following is probably well known

\begin{lem}
Let $\rho: \G\to \SO(V)$ be an almost effective representation descending to $\Bbb {HP}(V) = \HP^n$. Then the kernel of the action on $\Bbb {HP}(V)$ is contained in $\Sp(1)$.
\end{lem}

\begin{proof}
Identify  $V$ with the tangent space at a point $p\in \HP^{n+1}$. The representation gives rise to an isometric action on $\HP^{n+1}$ fixing $p$ and with the induced action on $\HP^{n}$ identified as the cut locus of $p$ in $\HP^{n+1}$. Since the subgroup of the isometry group of $\HP^{n+1}$ that fixes $\HP^{n}$ is $\Sp(1)$, the claim follows.
\end{proof}

\providecommand{\bysame}{\leavevmode\hbox    
to3em{\hrulefill}\thinspace}                                            

\end{document}